\def\ispreprint{1}        
\ifx\ispreprint\undefined
\documentclass[review,hidelinks,onefignum,onetabnum,dvipsnames]{siamart220329}
\else
\documentclass[DIV=14,dvipsnames,twoside=semi]{scrartcl}
\fi

\usepackage{array}
\usepackage{subfig}
\usepackage{xcolor}
\usepackage{xspace}
\usepackage{enumerate}
\usepackage{tikz}
\usepackage{pgfplots}
\usepackage{comment}
\usepackage{stmaryrd}

\usepackage{url}
\usepackage{hyperref}
\usepackage[hyphenbreaks]{breakurl}
\usepackage{pgfplots}
\pgfplotsset{compat=1.18}
\usepgfplotslibrary{groupplots}
\usetikzlibrary{arrows.meta}
\newcommand{\mylinewidth}{1pt}

\colorlet{color_ref}{Black}
\colorlet{color_rn}{ForestGreen}
\colorlet{color_sr3}{Red}
\colorlet{color_sr6}{RoyalBlue}
\colorlet{color_sr7}{BurntOrange}
\colorlet{color_sr8}{Plum}
\colorlet{color_sr10}{Brown}
\colorlet{color_sr12}{Violet}
\colorlet{color_sr15}{BlueGreen}

\def\markerref{p}
\def\markerrn{o}
\def\markersriii{pentagon}
\def\markersrvi{asterisk}
\def\markersrvii{square}
\def\markersrviii{diamond}
\def\markersrx{triangle}
\def\markersrxii{asterisk}
\def\markersrxv{asterisk}

\pgfplotscreateplotcyclelist{list_fig2}{%
  color=color_rn, solid, every mark/.append style={solid}, mark=\markerrn\\%
  color=color_sr3, solid, every mark/.append style={solid}, mark=\markersriii\\%
  color=color_sr6, dashed,every mark/.append style={solid}, mark=\markersrvi\\%
  color=color_sr7, solid, every mark/.append style={solid}, mark=\markersrvii\\%
  color=color_sr8, dashed, every mark/.append style={solid}, mark=\markersrviii\\%
  color=color_sr10, dotted, every mark/.append style={solid}, mark=\markersrx\\%
}

\pgfplotscreateplotcyclelist{list_fig3}{%
  color=color_ref, solid, every mark/.append style={solid}, mark=\markerref\\%
  color=color_rn, solid, every mark/.append style={solid}, mark=\markerrn\\%
  color=color_sr3, solid, every mark/.append style={solid}, mark=\markersriii\\%
  color=color_sr6, dashed,every mark/.append style={solid}, mark=\markersrvi\\%
  color=color_sr7, solid, every mark/.append style={solid}, mark=\markersrvii\\%
  color=color_sr8, dashed, every mark/.append style={solid}, mark=\markersrviii\\%
  color=color_sr10, dotted, every mark/.append style={solid}, mark=\markersrx\\%
}

\pgfplotscreateplotcyclelist{list_fig4}{%
  color=color_ref, solid, every mark/.append style={solid}, mark=\markerref\\%
  color=color_rn, solid, every mark/.append style={solid}, mark=\markerrn\\%
  color=color_sr3, solid, every mark/.append style={solid}, mark=\markersriii\\
  color=color_sr8, dashed, every mark/.append style={solid}, mark=\markersrviii\\%
  color=color_sr12, dotted, every mark/.append style={solid}, mark=\markersrxii\\%
  color=color_sr15, dotted, every mark/.append style={solid}, mark=\markersrxv\\%
}

\usetikzlibrary{calc, positioning}
\usetikzlibrary{decorations.pathreplacing, calligraphy, arrows, arrows.meta}

\usepackage{amsfonts}
\usepackage{graphicx}
\usepackage{epstopdf}
\usepackage{algorithmic}
\ifpdf
  \DeclareGraphicsExtensions{.eps,.pdf,.png,.jpg}
\else
  \DeclareGraphicsExtensions{.eps}
\fi

\ifx\ispreprint\undefined
\newsiamremark{remark}{Remark}
\newsiamremark{hypothesis}{Hypothesis}
\crefname{hypothesis}{Hypothesis}{Hypotheses}
\else
\usepackage{amsthm}
\theoremstyle{plain}
\newtheorem{lemma}{Lemma}
\newtheorem{proposition}{Proposition}
\newtheorem{theorem}{Theorem}
\theoremstyle{definition}
\newtheorem{definition}{Definition}
\theoremstyle{remark}
\newtheorem{remark}{Remark}

\usepackage{mathtools}
\fi

\DeclareMathOperator{\fl}{fl}
\DeclareMathOperator{\op}{op}
\DeclareMathOperator{\cov}{Cov}

\newcolumntype{R}[1]{>{\raggedleft\arraybackslash }b{#1}}
\newcolumntype{L}[1]{>{\raggedright\arraybackslash }b{#1}}
\newcolumntype{C}[1]{>{\centering\arraybackslash }b{#1}}

\newcommand{\abs}[1]{\left\lvert#1\right\rvert}
\newcommand{\norm}[1]{\left\lVert#1\right\rVert}
\newcommand{\SR}{\text{SR}}

\newcommand{\thetitle}{Probabilistic error analysis of\\%
limited-precision stochastic rounding}
\newcommand{\theshorttitle}{Limited-Precision Stochastic Rounding}
\newcommand{\thetitlefootnote}{Version of \today.}
\newcommand{\thefunding}{This work was funded by the HOLIGRAIL project (ANR-23-PEIA-0010).}

\newcommand{\theabstract}{Classical probabilistic rounding error analysis is particularly well suited to stochastic rounding (SR), and it yields strong results when dealing with floating-point algorithms that rely heavily on summation. For many numerical linear algebra algorithms, one can prove probabilistic error bounds that grow as $\mathcal{O}(\sqrt{n}u)$, where $n$ is the problem size and $u$ is the unit roundoff. These probabilistic bounds are asymptotically tighter than the worst-case ones, which grow as $\mathcal{O}(nu)$. For certain classes of algorithms, SR has been shown to be unbiased. However, all these results were derived under the assumption that SR is implemented exactly, which typically requires a number of random bits that is too large to be suitable for practical implementations. We investigate the effect of the number of random bits on the probabilistic rounding error analysis of SR. To this end, we introduce a new rounding mode, limited-precision SR. By taking into account the number $r$ of random bits used, this new rounding mode matches hardware implementations accurately, unlike the ideal SR operator generally used in the literature. We show that this new rounding mode is biased and that the bias is a function of $r$. As $r$ approaches infinity, however, the bias disappears, and limited-precision SR converges to the ideal, unbiased SR operator. We develop a novel model for probabilistic error analysis of algorithms employing SR. Several numerical examples corroborate our theoretical findings.}

\newcommand{\thekeywords}{Rounding error analysis,
    stochastic rounding,
    low-precision,
    floating-point arithmetic,
    inner product
}

\newcommand{\themsccodes}{
    65G50, 
    65F30. 
}

\newcommand{\theauthori}{El-Mehdi El arar}
\newcommand{\theemaili}{el-mehdi.el-arar@inria.fr}

\newcommand{\theauthorii}{Massimiliano Fasi}
\newcommand{\theemailii}{m.fasi@leeds.ac.uk}

\newcommand{\theauthoriii}{Silviu-Ioan Filip}
\newcommand{\theemailiii}{silviu.filip@inria.fr}

\newcommand{\theauthoriv}{Mantas Mikaitis}
\newcommand{\theemailiv}{m.mikaitis@leeds.ac.uk}

\newcommand{\theaffiliationi}{%
Université de Rennes, Inria, CNRS, IRISA, F-35000 Rennes, France
}

\newcommand{\theaffiliationii}{%
School of Computing, University of Leeds, Woodhouse Lane, Leeds LS2 9JT, UK
}

\newcommand{\theshortauthors}{E. M. El arar, M. Fasi, S.-I. Filip, and M. Mikaitis}

\usepackage{amsopn}

\ifpdf
\hypersetup{
  pdftitle={\theshorttitle},
  pdfauthor={\theshortauthors}
}
\fi

\usepackage{tcolorbox}
\tcbuselibrary{skins}
\tcbuselibrary{breakable}
\tcbset{shield externalize}

\def\F{\ensuremath{\mathcal{F}}}

\ifx\ispreprint\undefined
\headers{\theshorttitle}{\theshortauthors}
\else
\usepackage{scrlayer-scrpage}
\setkomafont{pageheadfoot}{\sffamily}
\cohead{\theshorttitle}
\cehead{\theshortauthors}
\ohead{\thepage}
\ifoot{}
\ofoot{}
\date{\vskip -30pt}
\hypersetup{colorlinks,linkcolor=Blue,citecolor=ForestGreen,urlcolor=Blue}
\usepackage{cleveref}
\newcommand{\funding}[1]{\textbf{Funding:} #1}
\newcommand{\email}[1]{\href{#1}{#1}}
\fi

\renewcommand{\textcolor}[2]{#2}
\renewcommand{\color}[1]{}

\begin{document}

\title{\thetitle\thanks{\thetitlefootnote \funding{\thefunding}}}

\author{%
    \theauthori\thanks{\theaffiliationi (\email{\theemaili, \theemailiii}).}
    \and
    \theauthorii\thanks{\theaffiliationii (\email{\theemailii, \theemailiv}).}
    \and
    \theauthoriii\footnotemark[2]
    \and
    \theauthoriv\footnotemark[3]}

\maketitle

\ifx\ispreprint\undefined
\begin{abstract}
    \theabstract
\end{abstract}


\else
\paragraph{Abstract.} \theabstract{}


\paragraph{MSC classification} \themsccodes{}
\fi

\section{Introduction}

A finite-precision number system cannot represent all reals, and values that cannot be represented must be approximated by a nearby machine-representable number.
This process, known as \emph{rounding}, ensures that partial results that are not representable can be used in subsequent limited-precision computations.
Rounding operators usually make a choice between two rounding candidates, one larger and one smaller than the value that has to be rounded, and the rules for making that choice are called \emph{rounding modes}.

Most rounding modes available in computer hardware are deterministic, meaning that for the same value, they will always return the same rounding candidate.
For example, round-to-nearest (RN) will always choose a machine-representable number that minimizes the distance from the input.
Here, we are concerned with a non-deterministic rounding mode: mode 1 \emph{stochastic rounding} (SR)~\cite{cfhm22}.
SR guarantees that rounding errors are random, which makes them more likely to cancel out. This rounding mode is particularly beneficial in alleviating \emph{stagnation}, a numerical phenomenon that causes small addends to be lost to rounding in long chains of sums.

SR can be traced back to the 1950s~\cite{fors50}, but it has seen a resurgence of interest in recent years.
It has been shown to be particularly well-suited to machine learning computations~\cite{gagn15,hofa92a}, and it is starting to become available in low-precision floating-point hardware: {\color{blue} both the Tesla D1~\cite{tesl-dojo} and AMD MI300 GPUs~\cite[sect.~7.2]{amd-mi300} support SR down-conversion of \texttt{binary32} values to 8-bit floating-point representations, whereas the Graphcore Intelligence Processing Unit (IPU)~\cite[sect.~7]{graph-mpai-320} offers the option to stochastically round (intermediate) \texttt{binary32} accumulation results in their Colossus IPU dot product operations to \texttt{binary16}. Starting with the Blackwell architecture, NVIDIA GPUs support SR down-conversion from \texttt{binary32} to 16, 8, 6 and 4-bit values~\cite[sect.~9.7.9.21]{nvidia-ptx8.7_isa}. The NeuronCore-v2 architecture powering AWS Trainium chips~\cite{fzg24} offers the possibility of using either RN or SR when doing \texttt{binary16} or \texttt{bfloat16} arithmetic (at least additions) by setting a global rounding flag.}


Indeed, artificial intelligence has motivated much of the recent research on SR, as this rounding mode has the potential to provide higher accuracy in various deep learning and optimization applications. When training neural networks with deterministic rounding modes, for example, the accumulation of gradients during parameter updates may be prone to stagnation, as the gradients become smaller in magnitude as training progresses. Gupta et al.~\cite{gagn15} have shown that SR can alleviate this issue, and this has been further explored in subsequent work~\cite{wcbc18,zzad20}.

SR has found some important uses in applications outside machine learning. It can alleviate stagnation in the solution of partial differential equations (PDEs), via Runge--Kutta finite difference methods in low precision~\cite{crgi22}, and of ordinary differential equations (ODEs), via the Euler, midpoint, and Heun methods~\cite{fami21}.
Paxton et al.~\cite{pcks22} have studied the effectiveness of low-precision arithmetic for climate simulations, focusing in particular on the effects of RN and SR in Lorenz systems, chaotic systems of ODEs, and PDEs related to climate modeling.
They have found that SR can effectively mitigate rounding errors across various applications, and these results provide evidence that SR may be relevant to next-generation climate models. For a comprehensive overview of the uses of SR, we refer the reader to~\cite[sect.~8]{cfhm22}.


Theoretical results have confirmed the beneficial effects of SR observed in numerical applications. Assuming precision-$p$ arithmetic, Parker~\cite{park97} has shown that SR is unbiased. Connolly et al.~\cite{chm21} have proven that SR satisfies the mean independence property, whereby one can construct martingales. Using the Azuma--Hoeffding inequality, they introduce a model that leads to probabilistic bounds on the relative error of a large class of linear algebra algorithms. Consider, for example, the evaluation of the inner product between two length-$n$ vectors in precision-$p$ floating-point arithmetic. The deterministic worst-case bound on the roundoff error that can occur during this computation grows as $nu$, where $u=2^{-p}$. Connolly et al.~\cite{chm21} have obtained a bound in $\mathcal{O}(\sqrt{n \ln{(n)}}u)$, and by using an alternative construction for the martingale, Ipsen and Zhou~\cite{ipzh20} have been able to derive a tighter one in $\mathcal{O}(\sqrt{n}u)$. This construction has been extended to the Horner method for polynomial evaluation~\cite{esop22}. A more recent approach relies exclusively on the variance of the error and the Bienaymé--Chebyshev inequality~\cite{esop23a}. The advantages of each approach are discussed in~\cite[sect.~4.3]{thesisarar}.

The rounding errors produced by $\SR$ are not always unbiased; examples of this are two algorithms for computing the sample variance~\cite{esop23}. Nevertheless, by using the aforementioned alternative martingale construction~\cite{esop23a,ipzh20},  one can derive probabilistic error bounds in $\mathcal{O}(\sqrt{n}u)$ for these two algorithms. Dexter et al.~\cite{dbmi24} show that, when rounding tall-and-thin matrices to low precision, if SR is used, then the smallest singular value of the rounded matrix is bounded away from zero with high probability.

SR implementations, both in hardware and in software, tend to use a summation-based algorithm~\cite[sect.~7]{cfhm22}. In a precision-$p$ floating-point number system (see~\Cref{sec:back} for details), this amounts to adding randomly generated bits to the trailing bits of the significand to be rounded. Ideally, these random bits should cover all the trailing bits to be rounded off, but this may be unfeasible, as the exact significand may have a non-terminating binary expansion. Therefore, the number $r$ of random bits to use must be limited in practice. This introduces a certain bias in the behavior of the implementations, with too small values of $r$ negating the expected benefits derived from using SR.

Ali,~Filip,~and~Sentieys~\cite{afs24} study this empirically in the context of deep neural network training, where they use SR for inner product computations in matrix multiplication. In the context of ODE solvers, experiments varying $r$ were performed for fixed-point arithmetic, where SR was used in the multiplication of fixed-point numbers~\cite[sect.~5c]{hmlf20}.

As far as we are aware, however, theoretical work on SR does not put constraints on $r$, and we are the first to consider the effects of limited precision from a theoretical point of view. To fully understand the behavior of practical implementations of SR, however, it is imperative to understand the impact of $r$ from an error analysis perspective. Therefore, we extend the analysis of standard algorithms under SR by taking into account the number of random bits used to perform the rounding. We call this new rounding mode \emph{limited-precision stochastic rounding}, and we denote it by $\SR_{p,r}$. In this notation, $p-1$ denotes the number of fraction bits available in the format to which the operator rounds, and similarly, we denote by $\SR_p$ the classical SR operator, for which $r = \infty$. We take the unit roundoff for precision-$p$ arithmetic to be $u_p:=2^{1-p}$, in line with previous work~\cite{esop23a}. Our main contributions are outlined below.
\begin{itemize}
    \item In \Cref{sec:limited-precision-sr}, we show that the unbiased and mean independence properties of $\SR_p$ break down for $\SR_{p,r}$.
    \item In \Cref{sec:limited-precision-sr,sec:applications}, we present a model whereby we can analyze algorithms under $\SR_{p,r}$ and compute probabilistic error bounds \textcolor{blue}{proportional to $\sqrt{n}u_p + nu_{p+r}$}, where $p$ is the working precision, $r$ the number of random bits, and $u_p=2^{1-p}$. 
    \item In \Cref{sec:applications}, we use our model to analyze the recursive summation and inner product algorithms.
    \item In \Cref{sec:bound-analysis}, we suggest a theoretically sound rule of thumb for choosing $r$ \textcolor{blue}{in recursive summation and inner product computations}:  setting $r=\lceil (\log_2n) / 2\rceil$ offers, with high probability, a good tradeoff between accuracy and operator complexity.
    \item In \Cref{sec:experiments}, we showcase the value of our analysis by means of numerical experiments focused on varying $r$.
\end{itemize}


\section{Notation and definitions}
\label{sec:back}
We recall basic definitions and properties from probability theory~\cite{miup05} and floating-point arithmetic~\cite{high02,mbdj18}.
We also formally introduce SR and its limited-precision variant.

\subsection{Probability}

\newcommand{\EE}{\ensuremath{\mathbb{E}}}
\newcommand{\VV}{\ensuremath{\mathbb{V}}}
Let $X$ be a random variable. Throughout this paper, $\EE(X)$ denotes the expected value of $X$ and $\VV(X)$ denotes its variance.
The conditional expectation of $X$ given $Y$ is denoted by $\EE(X \mid Y)$.

The following result provides a bound on the probability of a random variable falling within a specific number of standard deviations from its mean.

\begin{lemma}[Bienaymé--Chebyshev inequality]
\label{lem:bien-cheb-ineq}
Let $X$ be a random variable. If $\EE(X)$ and $\VV(X)$ are finite, then for any real number $\alpha > 0$, one has
$$ \mathbb{P}\big(\lvert X - \EE(X) \rvert \leq \alpha \sqrt{\VV(X)}\big) \geq 1- \frac{1}{\alpha^2}.
$$
\end{lemma}

To analyze how rounding errors accumulate, we will model them as random variables, and we will use probability to make precise statements about how rounding errors interact.

The linear relationship between two random variables $X$ and $Y$ is captured by their \emph{covariance}
\begin{equation*}
    \cov(X, Y) = \EE(XY) - \EE(X)\EE(Y).
\end{equation*}

If $\cov(X, Y) = 0$, then the two random variables are \emph{uncorrelated}, which means that there is no linear dependence between them.
The lack of linear correlation is not sufficient in our analysis, and we need to rely on the following property instead.

\begin{definition}[Mean independence]
	The random variable $X$ is \emph{mean independent} of the random variable $Y$ if $\EE(X \mid Y)=\EE(X)$. The  sequence of random variables $X_0, X_1,\ldots$ is \emph{mean independent} if $\EE(X_k \mid X_0, X_1, \ldots, X_{k-1}) = \EE(X_k)$ for all $k$.
\end{definition}

As the following result demonstrates, being mean independent is a stronger property than being uncorrelated but weaker than being truly independent.

\begin{proposition}
	Let $X$ and $Y$ be random variables.
	\begin{enumerate}
		\item If $X$ and $Y$ are independent, then $X$ is mean independent of $Y$.
		\item If $X$ is mean independent of $Y$, then $X$ and $Y$ are uncorrelated.
	\end{enumerate}
	The reciprocals of these two implications are false.
\end{proposition}

\textcolor{blue}{When $X$ is fully determined by $Y$, the conditional expectation~\cite[sect.~34]{billingsley1995probability} satisfies the following property.}
\begin{proposition}
\textcolor{blue}{Let $X$ and $Y$ be random variables defined on the same probability space. If $X$ is entirely determined by $Y$, meaning there exists a measurable function $g$ such that $X = g(Y)$, then the conditional expectation of $X$ given $Y$ satisfies $\mathbb{E}(X \mid Y) = X$.
}
\end{proposition}

\begin{definition}[{\cite[p.~295]{miup05}}]
	\label{def:martingale}
	A sequence of random variables $M_0,\ldots, M_n$ is a \emph{martingale} with respect to the sequence $X_0,\ldots, X_n$ if, for all $k,$
	\begin{itemize}
		\item $M_k$ is a function of $X_0,\ldots, X_k$,
		\item $\EE(\abs{M_k})$ is finite, and
		\item $\EE(M_k \mid  X_0,\ldots, X_{k-1})=M_{k-1}$.
	\end{itemize}
\end{definition}

\begin{lemma}[Azuma--Hoeffding inequality, {\cite[p.~303]{miup05}}]
\label{lem:azuma}
	Let $M_0,\ldots, M_n$ be a martingale with respect to a sequence $X_0,\ldots, X_n.$ We assume that there exist $0<a_k$ such that $-a_k \leq M_k - M_{k-1} \leq  a_k$ for $k = 1,\ldots,n.$ Then,  we have
	$$ \mathbb{P}\left( \lvert M_n - M_0 \rvert \leq \sqrt{\sum_{k=1}^n a_k^2} \sqrt{2 \ln (2 / \lambda)} \right) \geq 1- \lambda,
	$$
	where $0< \lambda <1$.
\end{lemma}

The mean independence property is essential to improve the error analysis of algorithms performed using stochastic rounding. It allows one to obtain a martingale (\Cref{def:martingale}), and then by applying \Cref{lem:azuma}, derive probabilistic bounds on the error that grows in $\mathcal{O}(\sqrt{n}u_p)$.

\subsection{Floating-point arithmetic}
\label{sec:FP}


Any \textcolor{blue}{nonzero} real number $x \in \mathbb{R}$ can be represented as
\begin{equation}
    x = (-1)^s \cdot 2^e \cdot m,\qquad
    s \in \{0,1\},\quad
    e \in \mathbb{Z},\quad
    m \in [1, 2).
\end{equation}
This representation, which we shall call \emph{scientific notation}, is unique for any $x \neq 0$.

Let $p > 0$. The set $\F \subset \mathbb{R}$ is a normalized precision-$p$ binary floating-point number system if it contains all real numbers $x$ that can be written as
$$x = (-1)^s \cdot 2^e \cdot m,$$
where the sign $s$ is either 0 or 1, the exponent $e$ is an integer in some range $[e_{\min}, e_{\max}]$, with $e_{\min} < e_{\max}$, and the significand $m \in [1, 2)$ is a real number with at most $p$ significant binary digits (bits). 
Customary choices of $p$ are defined in the IEEE 754 standard for floating-point arithmetic~\cite{ieee19}, the OCP 8-bit floating-point specification~\cite{modc23}, and the IEEE P3109 standard\footnote{\href{https://sagroups.ieee.org/p3109wgpublic/}{https://sagroups.ieee.org/p3109wgpublic/}} on arithmetic formats for machine learning.

For $x\in \mathbb{R}$, we denote the smallest precision-$p$ floating-point number no smaller than $x$ by $\llceil x \rrceil_p $, and the largest floating-point number no greater than $x$ by $\llfloor x \rrfloor_p$. In other words, we have
$$ \llceil x \rrceil_p=\min\{y\in \mathcal{F} : y \geq x\},  \quad \llfloor x \rrfloor_p=\max\{y\in \mathcal{F} : y \leq x\},
$$
and by definition, $\llfloor x \rrfloor_p  \leq x \leq \llceil x \rrceil_p$, with equality throughout if and only if $x \in\mathcal{F}$.
A real number $x\ne 0$ has two possible representations in floating-point arithmetic, $\llfloor x \rrfloor_p$ or $\llceil x \rrceil_p$, which coincide if $x \in \mathcal{F}$.
The operation that maps $x$ to the chosen candidate $\widehat x$ is called rounding, and the rounded quantity satisfies
\begin{equation}
    \widehat{x} =x(1+\delta), \label{fl(x)}
\end{equation}
where the relative error $\delta = (\widehat{x} - x)/x$ is such that $\abs{\delta} < 2^{1-p} = u_p$.
For round-to-nearest (RN) with any tie-breaking rule, we have the tighter bound $\abs{\delta} \leq u_p / (2+u_p)$~~\cite[Thm.~2.3]{mbdj18}.

\newcommand{\fpint}[1]{\ensuremath{\mathrm{ulp}_p(#1)}}
For $x \in \mathbb{R}$, we define the \emph{unit in the last place} as $\fpint{x} = 2^{e-p+1}$, where $e$ is the normalised exponent of $x$. Since $x = (-1)^s\cdot 2^e \cdot m$ with $1 \leq m < 2$, we have that $2^e \leq \lvert x \rvert < 2^{e+1}$ and
\begin{equation}\label{epsilon-bound}
    \fpint{x} = 2^{e} u_p
    \leq \abs{x} u_p.
\end{equation}
Let $x, y\in\mathcal F$ and $\op\in\{+, -, \times, \div\}$. We assume the standard model of floating-point arithmetic, whereby the error in one elementary operation is bounded, and we have
\begin{equation}
     \fl_p(x\; \op\; y) = (x\; \op\; y)(1+\delta),\qquad
     \fl_p(\sqrt{x}) = \sqrt{x}(1 + \delta) \label{fl(xopy)},
\end{equation}
where $\abs{\delta}$ is bounded by $u_p/2$ for RN with any tie-breaking rule and by $u_p$ for directed and stochastic rounding modes.

\subsection{Stochastic rounding}
\label{sec:SR}

\begin{definition}[Stochastic rounding]\label{def:sr}
Let $x \in \mathbb R$. The precision-$p$ stochastic rounding of $x$ to $\mathcal F$ is the Bernoulli random variable
\begin{equation}\label{eq:sr}
    \textnormal{SR}_p(x) =
    \begin{cases}
        \llceil x \rrceil_p,   & \text{with probability\ } q(x), \\
        \llfloor x \rrfloor_p, & \text{with probability\ } 1 - q(x),
    \end{cases}
    \qquad
    q(x) = \dfrac{x - \llfloor x \rrfloor_p}{\fpint{x}}.
\end{equation}
\end{definition}

If $x \not\in \mathcal F$, then $\fpint{x} = \llceil x \rrceil_p - \llfloor x \rrfloor_p$, which is the distance between the two floating-point numbers enclosing $x$. {\color{blue} Substituting this quantity in \eqref{eq:sr} when $x\not\in \mathcal F$ allows us to recover previous definitions of $\SR_p$ from the literature~\cite{cfhm22}.}

\begin{remark}
    This definition corresponds to what is called \emph{mode 1} SR or \textit{SR-nearness}. \emph{Mode 2} SR consists of taking $q(x)=1/2$ if $x\notin \mathcal F$ and $q(x)=0$ otherwise. While not our focus here, mode 2 SR can be an effective tool for doing what is called \emph{stochastic arithmetic}~\cite{vign04,gjp15,gjp18}, allowing one to detect instabilities in numerical routines and to provide accuracy estimates of computed results.
\end{remark}

The quantities in this definition are represented pictorially in \Cref{fig:sr}.
Note that if $x \in \mathcal F$, then $q(x) = 0$ and $\SR_p(x) = \llfloor x \rrfloor_p = x$ with probability 1. More generally, for $x \in \mathbb{R}$ we have
\begin{align*}
    \EE(\SR_p(x)) &= q(x)\llceil x \rrceil_p +\big(1-q(x)\big)\llfloor x \rrfloor_p \\
    &=  q(x)(\llceil x \rrceil_p - \llfloor x \rrfloor_p) + \llfloor x \rrfloor_p =x.
\end{align*}

The definition of $q(x)$ in~\eqref{eq:sr} assumes that $x$ is known with infinite precision. We now want to investigate how the behavior of $\SR_p$ changes when an infinitely precise $x$ is not available
\textcolor{blue}{or not used for efficiency reasons}, and we only have $\fl_{p+r}(x)$, the binary representation of $x$ truncated to the first $p+r$ binary digits, for some positive integer $r$. We use the letter $r$ to denote the additional bits of precision as this quantity corresponds to the number of random bits used to perform SR, as we will see later. This situation is also represented pictorially in \Cref{fig:sr}.

\def\ticksep{0.2}
\def\intsep{0.2}
\def\arrowsep{0.2}
\def\rtwosep{0.2}
\def\compssep{0.3}

\def\barheight{0.4}
\def\barsep{0.5}
\def\barlength{2.8}
\def\hdis{1.9}

\begin{figure}[t]
  \centering
  \begin{tikzpicture}[every node/.style={
      minimum width=0pt,
      minimum height=0pt,
      inner sep=0pt},semithick]


    \node (A) at (0,0) {};
    \node (B) at (3,0) {};
    \node (C) at (4,0) {};
    \node (D) at (8,0) {};

    \draw ($(A)-(\compssep,0)$)--($(D)+(\compssep,0)$);

    \foreach \x in {A,B,C,D} 
    \draw ($(\x)+(0,-\ticksep)$) -- ($(\x)+(0,\ticksep)$);


    \node [above=\ticksep+0.1 of A] {$\llfloor x \rrfloor_p$};
    \node [above=\ticksep+0.1 of B] {$x$};
    \node [above=\ticksep+0.1 of C] {$\fl_{p+r}(x)$};
    \node [above=\ticksep+0.1 of D] {$\llceil x \rrceil_p$};

    \draw [Stealth-Stealth]($(A)+(0.5*\intsep,-2*\intsep)$) --
    ($(B)+(-0.5*\intsep,-2*\intsep)$);
    \node [below=2.6*\intsep of $(A)!0.5!(B)$] {\small$q(x) \cdot  \fpint{x}$};
    \draw [Stealth-Stealth]($(B)+(0.5*\intsep,-2*\intsep)$) --
    ($(D)+(-0.5*\intsep,-2*\intsep)$);
    \node [below=2.6*\intsep of $(B)!0.5!(D)$] {\small$\bigl(1-q(x)\bigr) \cdot \fpint{x}$};

    \draw [Stealth-Stealth]($(A)+(0.5*\intsep,-6*\intsep)$) --
    ($(C)+(-0.5*\intsep,-6*\intsep)$);
    \node [below=6.6*\intsep of $(A)!0.5!(C)$] {\small$q_r(x) \cdot \fpint{x}$};
    \draw [Stealth-Stealth]($(C)+(0.5*\intsep,-6*\intsep)$) --
    ($(D)+(-0.5*\intsep,-6*\intsep)$);
    \node [below=6.6*\intsep of $(C)!0.5!(D)$] {\small$\bigl(1-q_r(x)\bigr) \cdot \fpint{x}$};

    \draw [Stealth-Stealth]($(A)+(0.5*\intsep,-10*\intsep)$) --
    ($(D)+(-0.5*\intsep,-10*\intsep)$);
    \node [below=10.6*\intsep of $(A)!0.5!(D)$] {\small$\fpint{x}$};

  \end{tikzpicture}
  \caption{Quantities used in the definitions~\eqref{eq:sr} and~\eqref{eq:sr-imp}.}
  \label{fig:sr}
\end{figure}
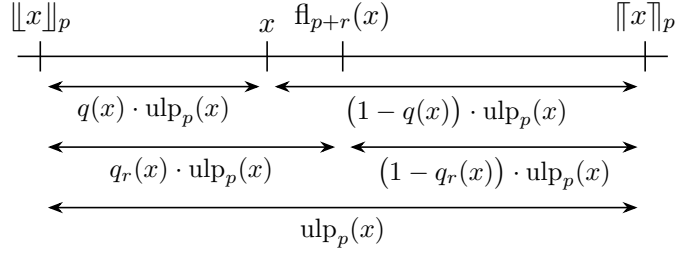

We note that $\fl_{p+r}$ could use a rounding mode other than truncation. All directed rounding modes satisfy $\fl_{p+r}(x) = x(1+\beta)$, where $\beta$ is a (deterministic) relative error such that $\abs{\beta} \leq 2^{1-p-r} = u_{p+r}$. For RN, the bound on $\abs{\beta}$ can be reduced by half, but this does not significantly affect our analysis.

\begin{definition}[limited-precision stochastic rounding]
    Let $x \in \mathbb R$. The limited-precision stochastic rounding using $r$ random bits $(\textnormal{SR}_{p,r})$ of $x$ to $\mathcal F$ is the Bernoulli random variable
\begin{equation}
    \label{eq:sr-imp}
    \textnormal{SR}_{p,r}(x) =
    \begin{cases}
        \llceil x \rrceil_p,   & \text{with probability\ } q_r(x), \\
        \llfloor x \rrfloor_p, & \text{with probability\ } 1- q_r(x),
    \end{cases}
    \qquad
    q_r(x) = \frac{\fl_{p+r}(x) - \llfloor x \rrfloor_p}{\fpint{x}}.
\end{equation}
\end{definition}

\begin{remark}
    \color{blue}{
    $\SR_{p,r}$ and $\SR_p$ both output $p$-bit precision results. The main difference is that $\SR_{p,r}$ uses $r$ random bits  during the operation before the truncation step, whereas $\SR_{p}$ essentially assumes an infinite value for $r$. Note that $\fl_{p+r}(x)$ is represented with $p+r$ bits of precision.}
\end{remark}

\begin{remark}
Depending on how $x$ is produced, one can determine the value of $r$ such that $\SR_{p,r}(x)=\SR_p(x)$. For instance, if $x=ab$ is the multiplication of two precision-$p$ values $a$ and $b$, then $x$ can be represented exactly as a precision-$2p$ value, meaning that $r=p$ random bits suffice. The situation is a bit different when $x=a \pm b$, since after normalization, $x$ can have a significand with up to $e_{\max}-e_{\min}+p-1$ bits, meaning that at most $r=e_{\max}-e_{\min}$ random bits are needed, which is indeed too large for common-use formats such as those specified in the IEEE 754 standard. \color{blue}{If $x = a / b$, then $x$ can have a non-terminating binary representation, meaning that we need an infinite number of random bits to perform $\SR_p$.} More generally, if $x$ is a precision-$q$ number with $q>p$, then it suffices to take $r=q-p$. 
\end{remark}

{\color{blue}{In actual hardware, the choice of $r$ varies among architectures. Whereas the Tesla~D1, AMD MI300 GPU, and Graphcore Colossus IPU take $r$ to be the minimal value needed to implement SR down-conversion exactly (i.e., it respects Definition~\ref{def:sr} when the input is a \texttt{binary32} value),
for the NVIDIA Blackwell architecture this is the case only for down-conversion to 16-bit formats. For 8, 6 and 4-bit down-conversion, a value of $r=16$ seems to always be used (see~\cite[sect.~9.7.9.21, Figs.~38--40]{nvidia-ptx8.7_isa}), which is not sufficient to respect Definition~\ref{def:sr}. There does not seem to be any public information regarding what value of $r$ is used in the Amazon Trainium chips.}}

\section{Properties of limited-precision stochastic rounding}
\label{sec:limited-precision-sr}
In this section, we demonstrate various properties of $\SR_{p,r}$ and, in particular, show that the mean independence property fails to hold. However, we prove a theorem that bounds the bias of the accumulated errors on algorithms with this rounding mode. Furthermore, using a technique similar to the one introduced in~\cite{esop23}, we present a model that allows us to compute probabilistic error bounds for algorithms using $\SR_{p,r}$.

Unlike $\SR_p$, the limited-precision $\SR_{p,r}$ operator is biased, since
\begin{align*}
    \EE\big(\SR_{p,r}(x)\big) &= q_r(x)\llceil x \rrceil_p +\big(1-q_r(x)\big)\llfloor x \rrfloor_p \\
    &=  q_r(x)(\llceil x \rrceil_p - \llfloor x \rrfloor_p) + \llfloor x \rrfloor_p\\
    &= \fl_{p+r}(x).
\end{align*}
Consequently, the bias $\EE\left(\SR_{p,r}(x) -x\right)$ is given by $\fl_{p+r}(x) -x$.
Let $\delta$ be such that $\SR_{p,r}(x) = x(1+\delta)$. Then
\begin{equation}
\label{eq:sr-bias}
    \EE(\delta) = \frac{\fl_{p+r}(x) -x}{x} = \beta.
\end{equation}
\begin{remark}
    \label{rem:SR_r-unbiased}
    With each error $\delta$ obtained with $\SR_{p,r}$, we have a corresponding deterministic error $\beta$ that arises from the computation of the function $\fl_{p+r}$.
    Moreover, for $x\notin \mathcal{F}$ and a large number of random bits $r$, $\fl_{p+r}(x)$ tends to $x$, which means that the rounding operator $\SR_{p,r}$ will become less and less biased \textcolor{blue}{as $r$ increases}.
\end{remark}

We now turn to bound the variance of the error in computation.
Let $x$ be the result of an elementary operation, and let $\SR_{p,r}(x) =x(1+\delta)$ that value rounded with $\SR_{p,r}$. If $x \in \mathcal{F}$, $\delta =0$, $\EE(\SR_{p,r} (x)) =x$, and $\VV(x)=0$. If not,  $\EE\big(\SR_{p,r}(x)\big)= \fl_{p+r}(x)$, and
\begin{align*}
 \VV\big(\SR_{p,r}(x)\big) &= \EE\big(\SR_{p,r}(x)^2\big) - \fl_{p+r}(x)^2 = \llceil x \rrceil_p^2 q_r(x) + \llfloor x \rrfloor_p^2 \big(1-q_r(x)\big) - \fl_{p+r}(x)^2 \\
 &= q_r(x) (\llceil x \rrceil_p^2 - \llfloor x \rrfloor_p^2) - (\fl_{p+r}(x)^2 - \llfloor x \rrfloor_p^2)\\
 &= q_r(x) \fpint{x} (\llceil x \rrceil_p + \llfloor x \rrfloor_p) - (\fl_{p+r}(x)-\llfloor x \rrfloor_p )(\fl_{p+r}(x)+ \llfloor x \rrfloor_p)\\
  &= q_r(x)\fpint{x} (\llceil x \rrceil_p + \llfloor x \rrfloor_p) - q_r(x) \fpint{x}(\fl_{p+r}(x)+ \llfloor x \rrfloor_p)\\
 &= q_r(x)\fpint{x} \big(\llceil x \rrceil_p - \fl_{p+r}(x) \big)\\
 &= \fpint{x}^2 q_r(x) \big(1-q_r(x)\big).
\end{align*}

Using~\eqref{epsilon-bound}, we obtain $\VV(\SR_{p,r}(x))\leq x^2 \frac{u_p^2}{4}$,
and
\begin{equation}\label{eq:delta}
    \VV(\delta) =
    \VV\left(\frac{\SR_{p,r}(x) -x}{x}\right)
    \leq \frac{u_p^2}{4}.
\end{equation}
The bound~\eqref{eq:delta} is analogous to that for $\SR_p$ discussed in~\cite[sect.~3]{esop23a} and, interestingly, does not depend on the number of random bits~$r$ used.

Let us consider an {\color{blue} execution of $k$ elementary} operations performed in floating-point arithmetic using $\SR_{p,r}$. {\color{blue} The following result shows that the mean independence property, which holds for $\SR_p$ and is shown in~\cite[lem.~5.2]{chm21}, does not extend to $\SR_{p,r}$. Moreover, the conditional expectation of the $k$th rounding error generated by an operation $a \textnormal{ op } b$ using $\SR_{p,r}$ is equal to the $k$th error produced when applying the rounding function $\fl_{p+r}$ to the infinitely precise result $c$ of the operation $a \textnormal{ op } b$.}

\begin{lemma}\label{lem:meanindp}
Let $a$ and $b$ be the result of $k-1$ scalar operations performed using limited-precision SR, and let $\delta_1, \ldots, \delta_{k-1}$ be the rounding errors these operations have produced. Let $c = a\;\op\;b$ for $\op \in\{+, -, \times, \div\}$, let $\beta_k$ be the deterministic relative rounding error defined by $\fl_{p+r}(c)=c(1+\beta_k)$ {\color{blue} corresponding to the truncation of $c$ to $p+r$ bits of precision}, and let $\widehat{c} = \textup{\SR}_{p,r}(c) = (a\;\op\;b) (1+\delta_k)$.
Then the random variable $\delta_k$ satisfies $\EE(\delta_k \mid  \delta_1,\ldots , \delta_{k-1}) = \beta_k$.
\end{lemma}

\begin{proof}
    The proof is analogous to that of~\cite[Lem.~5.2]{chm21}: the only difference is in the use of the probability $q_r(x)$ in \eqref{eq:sr-imp} instead of $q(x)$. If $c=0$, $\delta_k$ and $\beta_k$ are equal to $0$ and the result is immediate. Otherwise, the rounding error $\delta_k = (\widehat{c} - c)/c$ is a random variable that depends on $\delta_1,\ldots,\delta_{k-1}$ and has distribution
    $$
    \delta_k =
    \begin{cases}
        (\llceil c \rrceil_p - c)/{c},   &\text{with probability\ } q_r(c), \\
        (\llfloor c \rrfloor_p - c)/{c}, &\text{with probability\ } 1- q_r(c),
    \end{cases}
    \qquad
    q_r(c) = \frac{\fl_{p+r}(c) - \llfloor c \rrfloor_p}{\fpint{c}}.
    $$
    Moreover, $(\llceil c \rrceil_p - c)/c$ and $(\llfloor c \rrfloor_p - c)/c$ are themselves random variables that are determined
    by $\delta_1,\ldots,\delta_{k-1}$, which implies that
    \begin{equation*}
        \EE \left(\frac{\llceil c \rrceil_p - c}{c} \biggm| \delta_1,\ldots,\delta_{k-1} \right) = \frac{\llceil c \rrceil_p - c}{c},
        \quad
        \EE \left(\frac{\llfloor c \rrfloor_p - c}{c} \biggm| \delta_1,\ldots,\delta_{k-1} \right) = \frac{\llfloor c \rrfloor_p - c}{c}.
    \end{equation*}
    Therefore, we obtain
    \begin{align*}
    \EE(\delta_k \mid \delta_1,\ldots,\delta_{k-1})
    &=  q_r(c)\left(\frac{\llceil c \rrceil_p - c}{c}\right) + \big(1- q_r(c)\big) \left(\frac{\llfloor c \rrfloor_p - c}{c}\right)\\
    &=  q_r(c)\left(\frac{\llceil c \rrceil_p - c}{c} - \frac{\llfloor c \rrfloor_p - c}{c}\right) +  \frac{\llfloor c \rrfloor_p - c}{c}\\
    &=  q_r(c)\left(\frac{\llceil c \rrceil_p - \llfloor c \rrfloor_p}{c}\right) +  \frac{\llfloor c \rrfloor_p - c}{c}\\
    &= \frac{\fl_{p+r}(c) - \llfloor c \rrfloor_p}{c} +  \frac{\llfloor c \rrfloor_p - c}{c}\\
    &=  \frac{ \fl_{p+r}(c) -c}{c} = \beta_k.
    \end{align*}
\end{proof}

\begin{remark}
\label{rem:sr_r-cv-to-sr}
    From \Cref{rem:SR_r-unbiased}, if we assume that we have a large number of random bits $r$, the $\beta_1,\ldots,\beta_k$ will be equal to zero, and the $\delta_1,\ldots,\delta_k$ satisfy the mean independence property, which implies that the limited-precision SR operator $\SR_{p,r}$ coincides with the SR operator $\SR_p$.
\end{remark}

\begin{remark}
\label{rem:expectation}
    Note that, for $k \ge 2$, the value of $\beta_k$ depends on that of the random variables $\delta_1,\ldots,\delta_{k-1}$. Therefore, $\beta_k$ is a random variable despite having been produced by a deterministic rounding mode. From \Cref{lem:meanindp}, we have
    \begin{equation}
    \label{eq:delta=beta}
        \EE(\beta_k) = \EE\big(\EE(\delta_k \mid \delta_1,\ldots,\delta_{k-1})\big)= \EE(\delta_k).
    \end{equation}
\end{remark}

\begin{remark}
    Equation~\eqref{eq:delta=beta} shows that under $\SR_{p,r}$, the rounding error of an operation, $\delta$, and the truncation error applied for this operation, $\beta$, have the same expected value.
\end{remark}

\textcolor{blue}{The following theorem helps in establishing lower and upper bounds on the expectation of the accumulated error $\prod_{k= 1}^{n} (1+\delta_k)$ in an algorithm, where $\delta_1, \delta_2, \ldots, \delta_n$ are random errors obtained using $\SR_{p,r}$. }

\begin{theorem}
\label{thm:sr-mean}
\textcolor{blue}{
Let $\delta_1, \delta_2,\ldots, \delta_n$ and $\beta_1, \beta_2,\ldots, \beta_n$ be random variables such that $\EE(\delta_1) = \EE(\beta_1)$ and for all $2 \le k \le n$, $\EE\big(\delta_k \mid \delta_1,\ldots,\delta_{k-1}\big)= \beta_k$. Suppose there exists a constant $B> 0$ such that $\abs{\beta_k} \leq B$ for all $1 \le k \le n$.}
Then
\begin{equation*}
   \textcolor{blue}{(1-B)^{n}} \leq \EE\bigg(\prod_{k= 1}^{n} (1+\delta_k)\bigg)  \leq (1+B)^{n}.
\end{equation*}
\end{theorem}

\begin{proof}
    The proof is by induction on $n$. For $n=1$, we have $\EE(1+\delta_1) = 1+ \EE(\delta_1) = 1+ \EE(\beta_1)$, and
    $$\textcolor{blue}{(1-B)} \leq 1+ \beta_1 \leq (1+B).$$
    For the inductive step, let $Q_n =\prod_{k= 1}^{n} (1+\delta_k)$, and assume that the inequalities hold for $Q_{n-1}$. Since $Q_n = Q_{n-1} (1+\delta_n)$ \textcolor{blue}{and $\EE\bigl(\delta_n \mid  \delta_1,\ldots , \delta_{n-1}\bigr) = \beta_n$}, the law of total expectation $\EE(X)= \EE(\EE(X \mid Y))$ gives
    \begin{align*}
        \EE(Q_n) &= \EE\bigl( Q_{n-1} (1+\delta_n)\bigr) \\
        &= \EE\bigl(\EE(Q_{n-1} (1+\delta_n) \mid  \delta_1,\ldots , \delta_{n-1})\bigr) \\
        &= \EE\bigl(Q_{n-1} \EE( (1+\delta_n) \mid  \delta_1,\ldots , \delta_{n-1})\bigr) \\
        &=  \EE\bigl(Q_{n-1}(1+\beta_n)\bigr).
    \end{align*}
    \textcolor{blue}{Moreover $$
    \begin{cases} 
    \EE\bigl(Q_{n-1}(1+\beta_n)\bigr) \leq  \EE(Q_{n-1}) (1+B) \leq (1+B)^n, \\  
    \EE\bigl(Q_{n-1}(1+\beta_n)\bigr) \geq  \EE(Q_{n-1}) (1-B) \geq (1-B)^n.
    \end{cases}
    $$
    We thus have}
    \begin{align*}
        \textcolor{blue}{(1-B)^{n}} \leq \EE\bigg(\prod_{k= 1}^{n} (1+\delta_k)\bigg)  \leq (1+B)^{n}.
    \end{align*}

\end{proof}

\begin{remark}
    Let $\delta_1, \delta_2,\ldots, \delta_n$ and $\beta_1, \beta_2,\ldots, \beta_n$ as in \cref{lem:meanindp}. Then, by applying \cref{thm:sr-mean} we have:
    \begin{equation}
    \label{eq:model}
   \textcolor{blue}{(1-u_{p+r})^{n}} \leq \EE\bigg(\prod_{k= 1}^{n} (1+\delta_k)\bigg)  \leq (1+u_{p+r})^{n}.
\end{equation}
\end{remark}

\begin{remark}
    \textcolor{blue}{
    When using $\SR_p$,~\cite[Lem 6.1]{chm21} shows that $\EE\bigg(\prod_{k= 1}^{n} (1+\delta_k)\bigg) =1$. For $\SR_{p,r}$, taking $r \rightarrow \infty$ in~\eqref{eq:model} implies that $u_{p+r} \rightarrow 0$, which allows us to recover the same result as $\SR_p$.
    }
\end{remark}

\newcommand{\pset}[1][S]{\ensuremath{\mathcal P(#1)}}

In order to prove our main result, we need a technical lemma to rewrite a product of binomials as a sum of monomials. In the result, we denote by $\pset$ the power set of the set $S$, that is, the set whose elements are all possible subsets of $S$. We recall that $\pset$ can be defined recursively, since
\begin{equation}\label{eq:pset-def}
\begin{aligned}
\pset[\emptyset] &= \{\emptyset\},\\
\pset[S \cup \{\ell\}] &=
    \left\{
        T \cup \{\ell\} : T \in \pset[S]
    \right\} \cup \pset[S].
\end{aligned}
\end{equation}

\newcommand{\intinv}[1][m,n]{\ensuremath{\llbracket #1 \rrbracket_{\mathbb{N}}}}
\begin{lemma}
\label{lem:technical}
    Let $m,n \in \mathbb{N}$, with $m \le n$, let $\mathcal{I}= \{k \in \mathbb{N} : m \le k \le n\}$, and let $x_{k}, y_k\in \mathbb{R}$, for all $k \in \mathcal{I}$. Then, we have
    \begin{equation}\label{eq:tech-result}
    \prod_{k=m}^{n} (x_k+y_k)
    = \prod_{k = m}^{n} x_k +
    \sum_{\substack{K \in \mathcal{P}(\mathcal I)\\K \neq \mathcal I}}
    \left(\prod_{\vphantom{\mathcal I\setminus}i \in K} x_i \prod_{j \in \mathcal I \setminus K} y_j\right).
    \end{equation}
\end{lemma}

\begin{proof}
First, we prove by induction on the size of the index set $\mathcal I$ that
\begin{equation}\label{eq:tech-proof}
    \prod_{k=m}^{n} (x_k+y_k) =
    \sum_{K \in \mathcal{P}(\mathcal I)}
    \left(\prod_{\vphantom{\mathcal I\setminus}i \in K} x_i \prod_{j \in \mathcal I \setminus K} y_j\right),
    \qquad
    \mathcal I = \{k \in \mathbb{N} : m \le k \le n\}.
\end{equation}
\textcolor{blue}{
For $K= \emptyset$, we use the fact that $\prod_{\vphantom{\mathcal I\setminus}i \in K} x_i =1$.}
It is easy to verify that~\eqref{eq:tech-proof} holds for $m = n$. For the inductive step, assume that $\eqref{eq:tech-proof}$ holds. Without loss of generality, we extend the set by incrementing $n$, but one could prove the result analogously by decrementing $m$.
We have
\begin{align*}
    \prod_{k=m}^{n+1} (x_k+y_k)
    &= (x_{n+1} + y_{n+1}) \prod_{k=m}^{n} (x_k+y_k)\\
    &= (x_{n+1} + y_{n+1}) \sum_{K \in \mathcal{P}(\mathcal I)}
    \left(\prod_{\vphantom{\mathcal I\setminus}i \in K} x_i \prod_{j \in \mathcal I \setminus K} y_j\right),
    \
    \mathcal I = \{k \in \mathbb{N} : m \le k \le n\}\\
    &= x_{n+1} \sum_{K \in \mathcal{P}(\mathcal I)}
    \left(\prod_{\vphantom{\mathcal I\setminus}i \in K} x_i \prod_{j \in \mathcal I \setminus K} y_j\right)
    +  y_{n+1} \sum_{K \in \mathcal{P}(\mathcal I)}
    \left(\prod_{\vphantom{\mathcal I\setminus}i \in K} x_i \prod_{j \in \mathcal I \setminus K} y_j\right)\\
    &= \sum_{K \in \mathcal{P}(\mathcal I)}
    \left(\prod_{\vphantom{\mathcal I\setminus}i \in K  \cup \{n+1\}} x_i \prod_{j \in \mathcal I \setminus K} y_j\right)
    + \sum_{K \in \mathcal{P}(\mathcal I)}
    \left(\prod_{\vphantom{\mathcal I\setminus}i \in K} x_i \prod_{j \in \mathcal I \cup \{n+1\} \setminus K} y_j\right)\!.
\end{align*}
Note that the sum on the left is over the elements of $\{T \cup \{n+1\}: T \in \mathcal{P}(\mathcal I)\},$ and that on the right is over the elements of $\mathcal{P}$. Using the inductive step of~\eqref{eq:pset-def} with $\ell = n+1$, we obtain
\begin{equation*}
    \prod_{k=m}^{n+1} (x_k+y_k) =
    \sum_{K \in \mathcal{P}(\mathcal I')}
    \left(\prod_{\vphantom{\mathcal I\setminus}i \in K} x_i \prod_{j \in \mathcal I' \setminus K} y_j\right),
    \qquad
    \mathcal I' = \{k \in \mathbb{N} : m \le k \le n+1\}.
\end{equation*}
This establishes~\eqref{eq:tech-proof} for $\mathcal I'$.
To obtain~\eqref{eq:tech-result} from~\eqref{eq:tech-proof}, it suffices to extract the term $K = \mathcal I$, which corresponds to the product
\begin{align*}
    \prod_{k=m}^n x_k.
\end{align*}
\end{proof}

Equation~\eqref{eq:sr-bias} shows that $\SR_{p,r}$ is biased and
\Cref{lem:meanindp} demonstrates that $\SR_{p,r}$ does not satisfy the mean independence property. This complicates the direct theoretical analysis of $\SR_{p,r}$, since unbiasedness and mean independence are principal tools used to derive error bounds for algorithms using SR. However, it has been shown~\cite{esop23} that, for some algorithms, even when the use of SR leads to a biased result, it is possible to obtain probabilistic bounds in $\mathcal{O}(\sqrt{n}u_p)$. Building on similar techniques, we propose the following general model to study algorithms under $\SR_{p,r}$.

\begin{lemma}
    \label{lem:main-result}
    Let $\delta_1, \delta_2,\ldots, \delta_n$ be random errors produced by a sequence of elementary operations using $\textup{\SR}_{p,r}$, and let $\beta_1, \beta_2,\ldots, \beta_n$ be their corresponding errors incurred by $\fl_{p+r}$. Then, the random variables $\alpha_k = \delta_k - \beta_k$ for $1\leq k \leq n$ are mean independent, that is to say,
    $$  \EE(\alpha_k \mid \alpha_1,\ldots,\alpha_{k-1}) = \EE(\alpha_k) = 0.
    $$
    Moreover, for all $1\leq i \leq n$,
    \begin{equation}
         \prod_{k= i}^{n} (1+ \delta_{k}) = \prod_{k= i}^{n} (1+\alpha_{k}) +  \mathcal{B}_i,
    \end{equation}
    where $\mathcal{B}_i = \sum_{\substack{K \in \mathcal{P}(\mathcal I_i)\\ K \neq \mathcal I_i}}
        \left(\prod_{i \in K} (1 + \alpha_{i}) \prod_{j \in \mathcal I_i \setminus K} \beta_{j}\right)$ and $\mathcal I_i = \{k \in \mathbb{N} : i \le k \le n\}$  such that
        \begin{equation}
            \label{eq:b-bound}
            \abs{\mathcal{B}_i} \leq \gamma_{n-i+1}(u_p +u_{p+r}) -\gamma_{n-i+1}(u_p),
        \end{equation}
        where $\gamma_{m}(x)=(1+x)^{m} -1$.
\end{lemma}

\begin{proof}
    From~\eqref{eq:delta=beta}, we can conclude that $\EE(\alpha_k) = \EE(\delta_k) - \EE(\beta_k)=0$. Moreover, $\beta_k$ is entirely determined by $\delta_1,\ldots,\delta_{k-1}$, in particular $\alpha_1,\ldots,\alpha_{k-1}$.
\Cref{lem:meanindp} shows that $\EE(\delta_k \mid  \delta_1,\ldots , \delta_{k-1}) = \beta_k$, in particular, $ \EE(\delta_k \mid \alpha_1,\ldots,\alpha_{k-1})$ is also given by $\beta_k$. We thus have
    \begin{align*}
        \EE(\alpha_k \mid \alpha_1,\ldots,\alpha_{k-1}) &= \EE(\delta_k - \beta_k \mid \alpha_1,\ldots,\alpha_{k-1})\\
        &=  \EE(\delta_k \mid \alpha_1,\ldots,\alpha_{k-1}) - \EE(\beta_k \mid \alpha_1,\ldots,\alpha_{k-1})\\
        &= \beta_k - \beta_k =0.
    \end{align*}
    By applying \Cref{lem:technical} with $x_k = 1+\alpha_{k}$ and $y_k = \beta_{k}$ we obtain
     \begin{align*}
        \prod_{k= i}^{n} (1+\alpha_{k} + \beta_{k}) &=
        \prod_{k= i}^{n} (1+\alpha_{k}) +
        \sum_{\substack{K \in \mathcal{P}(\mathcal I_i)\\ K \neq \mathcal I_i}}
        \left(\prod_{i \in K} (1 + \alpha_{i}) \prod_{j \in \mathcal I_i \setminus K} \beta_{j}\right)\\
        &= \prod_{k= i}^{n} (1+\alpha_{k}) +  \mathcal{B}_i.
    \end{align*}
    Moreover
    \begin{align*}
      \abs{\mathcal{B}_i}
     &= 
        \sum_{\substack{K \in \mathcal{P}(\mathcal I_i)\\ K \neq \mathcal I_i}}
        \left(\prod_{i \in K} (1 + \alpha_{i}) \prod_{j \in \mathcal I_i \setminus K} \beta_{j}\right)\\
     &\leq 
        \sum_{\substack{K \in \mathcal{P}(\mathcal I_i)\\ K \neq \mathcal I_i}}
        (1 + u_p)^{\abs{K}} u_{p+r}^{n-i+1-\abs{K}}
        & \text{because\ } \abs{\alpha_{k}} \leq u_{p} \text{\ and\ } \abs{\beta_{k}} \leq u_{p+r}\\
    &= 
        \sum_{k=1}^{n-i+1}{n-i+1 \choose k}
        (1 + u_p)^{k} u_{p+r}^{n-i+1-k}
        & \\
     &= (1+u_p + u_{p+r})^{n-i+1} - (1+u_p)^{n-i+1}\\
     &= \gamma_{n-i+1}(u_p +u_{p+r}) -\gamma_{n-i+1}(u_p).&
\end{align*}

\end{proof}

\section{Error analysis of algorithms with limited-precision SR}
\label{sec:applications}
We are now ready to apply our results on limited-precision SR in order to analyze two common algorithms, recursive summation and inner product of vectors of floating-point numbers. We use \Cref{thm:sr-mean} to compute bounds on the biases of these algorithms. Furthermore, we use \Cref{lem:main-result} to compute probabilistic bounds of the relative errors of these algorithms under $\SR_{p,r}$. These bounds are established using two methods: 
\begin{itemize}
    \item martingales (\Cref{def:martingale}) and the Azuma--Hoeffding ineq.~(\Cref{lem:azuma}).
    \item the variance bound proved in~\cite[Lem.~3.1]{esop23a} and the Bienaymé--Chebyshev ineq.~(\Cref{lem:bien-cheb-ineq}). 
\end{itemize}
We show that these bounds are \textcolor{blue}{proportional to $\sqrt{n}u_p + nu_{p+r}$.}

\textcolor{blue}{An important observation is that all theorems proved in this section naturally converge to their established counterparts in the literature. More precisely, by letting $r$ tend to $\infty$ in our bounds, $u_{p+r} \to 0$, which recovers the results previously established for $\SR_p$, confirming the consistency of our study with existing theory.}

\subsection{Recursive summation}
\label{sub-sec:summation}
\newcommand{\condnum}[1]{\ensuremath{\kappa(#1)}}

Let $a \in \mathbb{R}^n$. We will now perform a roundoff error analysis for recursively computing the sum
\begin{equation}\label{eq:sum}
y = \sum_{i=1}^{n} a_i
\end{equation}
using precision-$p$ floating-point arithmetic with $\SR_{p,r}$.
In our analysis, we will rely on the condition number of the sum in~\eqref{eq:sum}, defined by
\begin{equation}\label{eq:cond-num-sum}
    \condnum{a} = \dfrac{\sum_{i=1}^{n} \abs{a_i}}{\abs{\sum_{i=1}^{n} a_i}},
\end{equation}
and on the error function
\begin{equation}\label{eq:gamma_n_u}
\gamma_n(u)= (1+u)^{n}-1 = n u + \mathcal{O}(u^2) \ \text{for} \ nu \ll 1.
\end{equation}

If we denote $\widehat{s}_k=\SR_{p,r}(\widehat{s}_{k-1}+a_k)$ for $k=2,\ldots,n$, we have
\vspace{0.2cm}
\begin{center}
	{\renewcommand{\arraystretch}{1.3}

		\begin{tabular}{| L{4.5cm}| L{3.5cm}| }
			\hline  Limited-precision SR & Exact computation  \\
			\hline    $\widehat{s}_1 = a_1  $ &  $s_1 = a_1 $  \\
			$\widehat{s}_{2}=(\widehat{s}_{1} + a_2) (1+\delta_{1}) $ &  $s_{2} = s_{1} + a_2 $\\
			$\widehat{s}_{k} = (\widehat{s}_{k-1} +a_{k})(1+\delta_{k-1}) $ &  $s_{k} = s_{k-1} +a_{k} $\\
			\hline  $\widehat{y}=\widehat{s}_{n} $ &  $y=s_{n} $\\
			\hline

		\end{tabular}
	}
\end{center}
\vspace{0.2cm}

It follows that
\begin{equation}
	\label{eq:summation}
	\widehat{y} = \sum_{i=1}^{n}
        \biggl(a_i \prod_{k= \max\{i,2\}}^{n} (1+\delta_{k-1})\biggr).
\end{equation}
Note that for $2\leq k \leq n$, one has $\abs{\delta_{k-1}} \leq u_p$ and $\abs{\EE(\delta_{k-1})}  = \abs{\EE(\beta_{k-1})}  \leq u_{p+r} $, where $\beta_{k-1}$ is defined analogously to \eqref{eq:sr-bias} as
\begin{equation}
\beta_{k-1} = \frac{\fl_{p+r}(\widehat{s}_{k-1} +a_k) -(\widehat{s}_{k-1}+a_k)}{\widehat{s}_{k-1} +a_k}.
\end{equation}
In the following theorem, we give a bound on the summation bias in \Cref{eq:summation}.

\begin{theorem}\label{thm:sum-expval}
    The quantity $\widehat y$ in \eqref{eq:summation} satisfies
    \begin{equation}
    \label{eq:error_sum}
    \frac{\abs{\EE(\widehat y) - y}}{\abs{y}} \leq \condnum{a} \gamma_{n-1}(u_{p+r}), 
    \end{equation}
    where $\condnum{a}$ and $\gamma_n(u_{p+r})$ are defined in~\eqref{eq:cond-num-sum} and \eqref{eq:gamma_n_u}, respectively.
\end{theorem}

\begin{proof}
    From \eqref{eq:summation}, we have
    \begin{align*}
        \abs{\EE(\widehat y ) -y}
        &= \abs{\EE\biggl(\sum_{i=1}^{n} a_i \prod_{k= \max\{i,2\}}^{n} (1+\delta_{k-1}) \biggr) - \sum_{i=1}^{n} a_i}\\
        &= \abs{\sum_{i=1}^{n} a_i \left(\EE\Biggl( \prod_{k= \max\{i,2\}}^{n} (1+\delta_{k-1}) \Biggr) -1\right)} && \text{by linearity}\\
        &\leq \sum_{i=1}^{n} \abs{a_i} \abs{\EE\left( \prod_{k= \max\{i,2\}}^{n} (1+\delta_{k-1}) \right) -1} && \text{by triangle inequality} \\
        &\leq \sum_{i=1}^{n} \abs{a_i} \bigl((1+u_{p+r})^{n-\max\{i,2\}+1} -1 \bigr) && \text{by \Cref{thm:sr-mean}} \\
        &\leq \sum_{i=1}^{n} \abs{a_i} \bigl((1+u_{p+r})^{n-1} -1 \bigr)\\
        &= \sum_{i=1}^{n} \abs{a_i} \gamma_{n-1}(u_{p+r}).
    \end{align*}
    We thus have
    \begin{align*}
        \frac{\abs{\EE(\widehat y) - y}}{\abs{y}} \leq \condnum{a}\gamma_{n-1}(u_{p+r}).
    \end{align*}
\end{proof}

We now turn to give a probabilistic bound for the relative error of the summation~\eqref{eq:summation} under $\SR_{p,r}$ using martingales and \Cref{lem:azuma}.

\begin{theorem}
    \label{thm:proba-sum}
    For any $0 < \lambda < 1$, the quantity $\widehat y$ in \eqref{eq:summation} satisfies
    \begin{equation}
    \label{eq:proba-error_sum}
     \frac{\abs{\widehat y - y}}{\abs{y}}
     \leq \condnum{a} \left(\sqrt{u_p \gamma_{2(n-1)}(u_p)} \sqrt{\ln (2 / \lambda)} + \gamma_{n-1}(u_p +u_{p+r}) -\gamma_{n-1}(u_p)\right),
    \end{equation}
    with probability at least $1-\lambda$.
\end{theorem}

\begin{proof}
    For all $1\leq i \leq n$, denote $\mathcal I_i = \{k \in \mathbb{N} : \max\{i,2\} \le k \le n\}$. \Cref{lem:main-result} shows that the random variables $\alpha_1, \alpha_2,\ldots, \alpha_{n-1}$, with $\alpha_j = \delta_j - \beta_j$ are mean independent and
    $$
        \prod_{k= \max\{i,2\}}^{n} (1+\delta_{k-1}) = \prod_{k= \max\{i,2\}}^{n} (1+\alpha_{k-1} + \beta_{k-1}) = \prod_{k= \max\{i,2\}}^{n} (1+\alpha_{k-1}) +  \mathcal{B}_i.
    $$
    We thus have
\begin{equation}\label{eq:two-pieces}
\begin{aligned}
    \abs{\widehat y - y} &= \abs{\sum_{i=1}^{n} a_i \left(\prod_{k= \max\{i,2\}}^{n} (1+\delta_{k-1}) -1 \right)}\\
    &= \abs{\sum_{i=1}^{n} a_i \left(\prod_{k= \max\{i,2\}}^{n} (1+\alpha_{k-1}) +  \mathcal{B}_i -1\right) }\\
    &\leq \abs{\sum_{i=1}^{n} a_i \left(\prod_{k= \max\{i,2\}}^{n} (1+\alpha_{k-1}) -1\right)} + \abs{\sum_{i=1}^{n} a_i \mathcal{B}_i}\\
    &= \abs{M} + \abs{A},
\end{aligned}
\end{equation}
where $M = \sum_{i=1}^{n} a_i (\prod_{k= \max\{i,2\}}^{n} (1+\alpha_{k-1}) -1)$, and $A= \sum_{i=1}^{n} a_i \mathcal{B}_i$.
\textcolor{blue}{Since the $\alpha_{k-1}$ are mean independent and bounded in magnitude by $u_p$, $M$ can form a martingale as previously established in the literature~\cite{esop23a,ipzh20}. We now obtain a probabilistic bound on this martingale that is proportional to $\sqrt{n}u_p$. We will only outline our chosen construction here, referring the reader to~\cite[chap.~4]{thesisarar} for a summary of possible methods.}

By~\Cref{lem:azuma}, we obtain the bound
\begin{equation}\label{eq:mart-bound}
   \abs{M} \leq \abs{\sum_{i=1}^{n} a_i} \sqrt{u_p \gamma_{2(n-1)}(u_p)} \sqrt{\ln (2 / \lambda)},
\end{equation}
which holds with probability at least $1-\lambda$. Let us bound the second term in~\eqref{eq:two-pieces}. For all $1\leq i \leq n$, because $\mathcal I_i \subset \mathcal I_2$, $\abs{\mathcal{B}_i} \leq \abs{\mathcal{B}_2}$ and \Cref{lem:main-result} shows that
\begin{align*}
      \abs{\mathcal{B}_i} &\leq \abs{\mathcal{B}_2}
     \leq \gamma_{n-1}(u_p +u_{p+r}) -\gamma_{n-1}(u_p).
\end{align*}
It follows that
\begin{align}
\label{eq:A}
    \abs{A}  = \abs{\sum_{i=1}^{n} a_i \mathcal{B}_i}
    \leq  \sum_{i=1}^{n} \abs{a_i} \abs{\mathcal{B}_i}
    \leq  \left(\sum_{i=1}^{n} \abs{a_i}\right)  \left( \gamma_{n-1}(u_p +u_{p+r}) -\gamma_{n-1}(u_p)\right).
\end{align}
Therefore, the inequality
$$ \frac{\abs{\widehat y - y}}{\abs{y}}  \leq \condnum{a} \left(\sqrt{u_p \gamma_{2(n-1)}(u_p)} \sqrt{\ln (2 / \lambda)} + \gamma_{n-1}(u_p +u_{p+r}) -\gamma_{n-1}(u_p)\right)
$$
holds with probability at least $1-\lambda$.
\end{proof}

In the following, we give a probabilistic bound on the relative error of the summation~\eqref{eq:summation} under $\SR_{p,r}$ using the bound on the variance proposed in~\cite[Lem.~3.1]{esop23a} and \Cref{lem:bien-cheb-ineq}.

\begin{theorem}
    \label{thm:bc-proba-sum}
    For any $0 < \lambda < 1$, the quantity $\widehat y$ in \eqref{eq:summation} satisfies
    \begin{equation}
    \label{eq:bc-proba-error_sum}
     \frac{\abs{\widehat y - y}}{\abs{y}} \leq \condnum{a} \left(\sqrt{\gamma_{n-1}(u_p^2) / \lambda} + \gamma_{n-1}(u_p +u_{p+r}) -\gamma_{n-1}(u_p)\right),
    \end{equation}
    with probability at least $1-\lambda$.
\end{theorem}

\begin{proof}
    This proof relies on~\eqref{eq:two-pieces}. Since $\alpha_{k-1}$ are mean independent and satisfy $\abs{\alpha_{k-1}} \leq u_p$,~\cite[Lem.~3.1]{esop23a} and \Cref{lem:bien-cheb-ineq} show that
    \begin{equation}
    \label{eq:var-M}
        \abs{M} \leq  \left(\sum_{i=1}^{n} \abs{a_i} \right) \sqrt{\gamma_{n-1}(u_p^2)/\lambda}
    \end{equation}
    holds with probability at least $1-\lambda$. Using~\eqref{eq:A}, we conclude that
    $$ \frac{\abs{\widehat y - y}}{\abs{y}}  \leq \condnum{a} \left(\sqrt{\gamma_{n-1}(u_p^2)/\lambda} + \gamma_{n-1}(u_p +u_{p+r}) -\gamma_{n-1}(u_p)\right),
$$
with probability at least $1-\lambda$.
\end{proof}

\subsection{Computation of inner products}
\label{sub-sec:IP}

Let $a,b \in \mathbb{R}^n$. We will now extend the analysis of summation to the computation of the inner product
\begin{equation}\label{eq:IP}
y = \sum_{i=1}^{n} a_i b_i.
\end{equation}
Our analysis will rely on the condition number of the inner product \condnum{a \circ b}, where $\condnum{\cdot}$ is defined in~\eqref{eq:cond-num-sum} and  $a \circ b = (a_1 b_1,\ldots,a_n b_n)$ denotes the Hadamard product. If we denote $$
\left\{
\begin{aligned}
    \widehat{s}_{2k-1} &= \widehat{s}_{2k-2} + \SR_{p,r}(a_k b_k), \\
    \widehat{s}_{2k} &= \SR_{p,r}(\widehat{s}_{2k-1}),
\end{aligned}
\quad \text{for } k = 1, \ldots, n,
\right.
$$

we have
\vspace{0.2cm}
\begin{center}
	{\renewcommand{\arraystretch}{1.3}

		\begin{tabular}{| L{5.5cm}| L{3.5cm}| }
			\hline  Limited-precision SR & Exact computation  \\
			\hline    $\widehat{s}_1 = a_1 b_1  $ &  $s_1 = a_1 b_1 $  \\
			$\widehat{s}_{2}=\widehat{s}_{1}  (1+\delta_{1}) $ &  $s_{2} = s_{1}  $\\
			$\widehat{s}_{2k-1} = \widehat{s}_{2k-2} +a_{k} b_k(1+\delta_{2k-2}) $ &         $s_{2k-1} = s_{2k-2} +a_{k} b_k $\\
                $\widehat{s}_{2k} = \widehat{s}_{2k-1}(1+\delta_{2k-1}) $ &  $s_{2k} = s_{2k-1} $\\
			\hline  $\widehat{y}=\widehat{s}_{2n} $ &  $y=s_{2n} $\\
			\hline

		\end{tabular}
	}
\end{center}
\vspace{0.2cm}

It follows that
\begin{equation}
	\label{eq:IP-SR}
	\widehat{y} =  \sum_{i=1}^{n} a_ib_i (1+\delta_{2i-1}) \prod_{k=i}^n (1+\delta_{2(k-1)}),
\end{equation}
with $\delta_0 = 0$. Note that for $1\leq k \leq 2n-1$, one has $\abs{\delta_{k}} \leq u_p$ and $\abs{\EE(\delta_{k})}  = \abs{\EE(\beta_{k})}  \leq u_{p+r} $, where $\beta_{k}$ is defined analogously to \eqref{eq:sr-bias} as
\begin{equation}
\beta_{2k-1} = \frac{\fl_{p+r}(\widehat{s}_{2k}) -(\widehat{s}_{2k-1})}{\widehat{s}_{2k-1} } \quad \text{and} \quad \beta_{2k-2} = \frac{\fl_{p+r}(\widehat{s}_{2k-1}) -(\widehat{s}_{2k-2}+a_k b_k)}{\widehat{s}_{2k-2} +a_k b_k}.
\end{equation}
In the following theorem, we give a bound on the inner product bias in \Cref{eq:IP-SR}.

\begin{theorem}\label{thm:IP-expval}
    The quantity $\widehat y$ in Equation~\eqref{eq:IP-SR} satisfies
    \begin{equation}
    \label{eq:error_IP}
    \frac{\abs{\EE(\widehat y) - y}}{\abs{y}} \leq \condnum{a \circ b} \gamma_n(u_{p+r}), 
    \end{equation}
    where $\condnum{a \circ b}$ and $\gamma_n(u_{p+r})$ are defined in~\eqref{eq:cond-num-sum} and \eqref{eq:gamma_n_u}, respectively.
\end{theorem}

\begin{proof}
    The proof follows the same structure as that of \Cref{thm:sum-expval},
    with one additional error term resulting from the multiplication $a_i b_i$.
\end{proof}

The following two theorems give probabilistic bounds on the relative error of the inner product~\eqref{eq:IP-SR} under $\SR_{p,r}$ using martingales and \Cref{lem:azuma} and the bound on the variance proposed in~\cite[Lem.~3.1]{esop23a} and \Cref{lem:bien-cheb-ineq}, respectively.

\begin{theorem}
    \label{thm:proba-IP}
    For any $0 < \lambda < 1$, the quantity $\widehat y$ in \eqref{eq:IP-SR} satisfies
    \begin{equation}
    \label{eq:proba-error_IP}
     \frac{\abs{\widehat y - y}}{\abs{y}}
     \leq \condnum{a \circ b} \left(\sqrt{u_p \gamma_{2n}(u_p)} \sqrt{\ln (2 / \lambda)} + \gamma_{n}(u_p +u_{p+r}) -\gamma_{n}(u_p)\right),
    \end{equation}
    with probability at least $1-\lambda$. 
\end{theorem}

\begin{proof}
    The proof follows the same structure as that of \Cref{thm:proba-sum},
    with one additional error term resulting from the multiplication $a_i b_i$.
\end{proof}

\begin{theorem}
    \label{thm:bc-proba-IP}
    For any $0 < \lambda < 1$, the quantity $\widehat y$ in \eqref{eq:IP-SR} satisfies
    \begin{equation}
    \label{eq:bc-proba-error_ip}
     \frac{\abs{\widehat y - y}}{\abs{y}} \leq \condnum{a \circ b} \left(\sqrt{\gamma_{n}(u_p^2) / \lambda} + \gamma_{n}(u_p +u_{p+r}) -\gamma_{n}(u_p)\right),
    \end{equation}
    with probability at least $1-\lambda$.
\end{theorem}

\begin{proof}
    The proof follows the same structure as that of \Cref{thm:bc-proba-sum},
    with one additional error term resulting from the multiplication $a_i b_i$.
\end{proof}

\begin{remark}
\label{rem:sr-main}
    We have shown the applicability of~\Cref{lem:main-result} to summation and inner product. These results \textcolor{blue}{should} hold for all algorithms considered for $\SR_p$ that satisfy a probabilistic bound in terms of $\mathcal{O}(\sqrt{n}u_p)$, \textcolor{blue}{in particular, algorithms with multi-linear errors~\cite[sect.~4]{thesisarar}. We plan to explore this in future work.}
\end{remark}

\section{Bound analysis}
\label{sec:bound-analysis}

We analyze the bound in \Cref{thm:proba-IP}; the same analysis holds for \Cref{thm:bc-proba-IP} and the probabilistic bounds in \Cref{thm:proba-sum} and \Cref{thm:bc-proba-sum}.

Let $f: \mathbb{R}^2 \rightarrow \mathbb{R}$ be twice continuously differentiable at $(0,0)$. By Taylor's theorem, we have
\begin{equation}\label{eq:multi-taylor}
f(x, y) = f(0, 0) + \frac{\partial f}{\partial x}(0, 0) x + \frac{\partial f}{\partial y}(0, 0) y + \mathcal{O}(\norm{(x,y)}_2).
\end{equation}
By using~\eqref{eq:multi-taylor}, we obtain
\begin{equation}\label{eq:bound1}
\begin{aligned}
    f(u_p,u_{p+r}) &:= \gamma_n(u_p + u_{p+r}) 
    = (1 + u_p + u_{p+r})^{n} - 1 \\
    &= n (u_p + u_{p+r}) + \mathcal{O}(\norm{(u_p, u_{p+r})}_2).
\end{aligned}
\end{equation}
Combining~\eqref{eq:bound1} with
\begin{equation*}
\gamma_n(u_p + u_{p+r}) = (1 + u_p + u_{p+r})^{n} - 1 = n (u_p + u_{p+r}) + \mathcal{O}(\norm{(u_p, u_{p+r})}_2)
\end{equation*}
and
\begin{equation*}
\sqrt{u_p \gamma_{2n}(u_p)} \sqrt{\ln (2 / \lambda)} = \sqrt{2n} \sqrt{\ln (2 / \lambda)} u_p + \mathcal{O}(u_p^2),
\end{equation*}
we obtain
\begin{align*}
    \frac{\abs{\widehat y - y}}{\abs{y}}  &\leq \condnum{a \circ b} \left(\sqrt{u_p \gamma_{2n}(u_p)} \sqrt{\ln (2 / \lambda)} + \gamma_{n}(u_p +u_{p+r}) -\gamma_{n}(u_p)\right)\\
    &= \condnum{a \circ b} \left(\sqrt{2n} \sqrt{\ln (2 / \lambda)} u_p + n (u_p + u_{p+r}) - n u_p + \mathcal{O}(\norm{(u_p, u_{p+r})}_2)  \right)\\
    &= \condnum{a \circ b} \left(\sqrt{2n} \sqrt{\ln (2 / \lambda)} u_p + n u_{p+r}\right) + \mathcal{O}(\norm{(u_p, u_{p+r})}_2).
\end{align*}

This result indicates that the bound on $\SR_{p,r}$ with fixed $\lambda$ is made up of two components: $\sqrt{n} u_p$, a probabilistic term that captures the random behavior of the algorithm, and $n u_{p+r}$, a deterministic one that captures the truncations performed in precision $p+r$. As the number of random bits $r$ increases, the bound becomes tighter as the magnitude of $u_{p+r}$ decreases. Moreover, the probabilistic bounds converge asymptotically to the probabilistic bound on $\SR_p$, which is consistent with~\Cref{rem:sr_r-cv-to-sr}.

\begin{remark}
    A good rule of thumb is to pick a value of $r$ that ensures that the term containing $\sqrt{n}u_p$ is not dominated by the one containing $n u_{p+r}$. Simplifying the two terms leads to $r\geq \lceil (\log_2n) / 2 \rceil$, where $\lceil x \rceil$ is the smallest integer greater or equal than $x\in\mathbb{R}$. \textcolor{blue}{Since our~\cref{lem:main-result} should be applicable to all algorithms with multi-linear errors, we expect this rule of thumb to work for all such algorithms.}
    \Cref{sec:experiments} looks at various numerical examples where such a value of $r$ is indeed a sensible choice.
\end{remark}

\section{Numerical experiments}\label{sec:experiments}
We perform a set of numerical experiments by considering computations that are prone to stagnation with RN, and we investigate the effect that the value of $r$ has when $\SR$ is used. The first two sets of experiments use the \texttt{srfloat} C++ library\footnote{\href{https://github.com/sfilip/srfloat}{https://github.com/sfilip/srfloat}}, also available with Python bindings, which simulates $\SR_{p,r}$ arithmetic as described in \Cref{sec:SR}. Internally, the software uses \texttt{binary64} arithmetic, and it can handle values of $p$ up to $53$ and values of $r$ up to $53-p$ for a chosen $p$. The neural network training example relies on the \texttt{mptorch}\footnote{\href{https://github.com/mptorch/mptorch}{https://github.com/mptorch/mptorch}} library, a PyTorch extension with similar $\SR_{p,r}$ simulation functionality for deep learning computations.

\subsection{Summation}

\begin{figure}[t]
  \centering
  \footnotesize
    \begin{tikzpicture}[trim axis group left,
        trim axis group right]
      \begin{groupplot}[
        group style={
          group size=2 by 1,
        },
        ymode=log,
        width=2.5in,
        grid=major,
        every axis plot/.append style={mark repeat=5,
            line width=\mylinewidth},
        ]

        \nextgroupplot[
        title={Relative forward error},
        xlabel = {$n$},
        ymax=1000,
        ymin=0.000001,
        cycle list name = list_fig2
        ]

        \addplot table [x=n, y=Error_rn, col sep=comma] {errs_rn.csv};

        \addplot table [x=n, y=Error_sr_3, col sep=comma] {errs_sr_3.csv};

        \addplot table [x=n, y=Error_sr_6, col sep=comma] {errs_sr_6.csv};

        \addplot table [x=n, y=Error_sr_7, col sep=comma] {errs_sr_7.csv};

        \addplot table [x=n, y=Error_sr_8, col sep=comma] {errs_sr_8.csv};

        \addplot table [x=n, y=Error_sr_10, col sep=comma] {errs_sr_10.csv};

        \nextgroupplot[
        title={Bounds with $1-\lambda=0.9$},
        xlabel = {$n$},
        ymax=1000,
        ymin=0.000001,
        cycle list name = list_fig2
        ]

        \addplot table [x=n, y=bound_rn, col sep=comma] {bound_rn.csv};

        \addplot table [x=n, y=bound_sr_3, col sep=comma] {bound_sr_3.csv};

        \addplot table [x=n, y=bound_sr_6, col sep=comma] {bound_sr_6.csv};

        \addplot table [x=n, y=bound_sr_7, col sep=comma] {bound_sr_7.csv};

        \addplot table [x=n, y=bound_sr_8, col sep=comma] {bound_sr_8.csv};

        \addplot table [x=n, y=bound_sr_10, col sep=comma] {bound_sr_10.csv};

      \end{groupplot}
    \end{tikzpicture}

    \smallskip

    \begin{tikzpicture}[trim axis left, trim axis right]
      \begin{axis}[
        title = {},
        legend columns=3,
        scale only axis,
        width=1mm,
        hide axis,
        /tikz/every even column/.append style={column sep=0.6cm},
        legend style={at={(0,0)},anchor=center,draw=none,
          legend cell align={left},cells={line width=\mylinewidth}},
        legend image post style={sharp plot},
        legend cell align={left},
        cycle list name = list_fig2
        ]
        \addplot (0,0);
        \addplot (0,0);
        \addplot (0,0);
        \addplot (0,0);
        \addplot (0,0);
        \addplot (0,0);
        \legend{RN, $\SR_{11,3}$, $\SR_{11,6}$, $\SR_{11,7}$, $\SR_{11,8}$, $\SR_{11,10}$};
      \end{axis}
    \end{tikzpicture}
  \caption{Left: relative error of RN and $\SR_{11,r}$ in IEEE-754 binary16 arithmetic ($p=11$) for the recursive summation of $n$ floating-point numbers drawn from a uniform distribution between 0 and 1. For each value of $n$, the reported relative error for $\SR_{11,r}$ is the average value over $500$ runs. Right: comparison of deterministic bound and probabilistic bounds \textup{(}\Cref{thm:bc-proba-sum}\textup{)} with the associated random bits.}
  \label{fig:cond-1}
\end{figure}
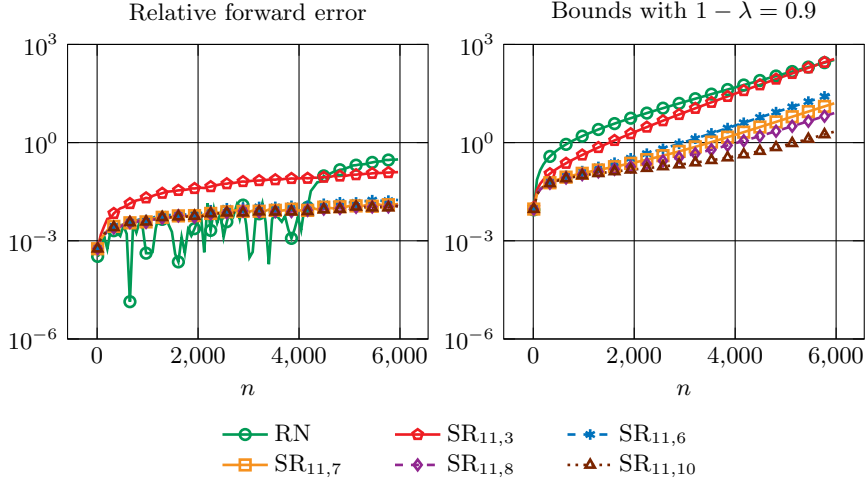

We first look at the effect of $r$ when recursively summing $n$ floating-point numbers, for $n$ between 2 and 6,000. The addends are sampled uniformly at random from the interval $[0,1]$, and in this experiment we use \texttt{binary16} arithmetic, for which $p = 11$. The results are reported in~\Cref{fig:cond-1}. The left panel shows relative errors when using RN and various $\SR_{11,r}$ rounding modes. While RN is superior for smaller $n$, as soon as the running sum becomes sufficiently large, new addends will tend to get absorbed, leading to stagnation. The biased errors that are introduced this way accumulate, leading to a significant increase in the overall error. While performing worse than RN in the beginning, the $\SR_{11,r}$ configurations are less affected by stagnation as $n$ increases. Taking $r$ close to $\lceil (\log_2 6{,}000) / 2\rceil=7$ gives good results, with larger values of $r$ not showing any significant improvements.

The figure on the right corroborates the findings presented in \Cref{sec:bound-analysis} and shows that the probabilistic bounds are tighter than deterministic ones. We use the bound obtained with the variance bound and the Bienaymé--Chebyshev inequality because, with probability $0.9$, this method ensures tight bounds for a larger $n$. For more details on the comparison between the bounds in \cref{thm:proba-sum} and \Cref{thm:bc-proba-sum}, we refer the reader to~\cite[sect 4.3]{thesisarar}.
The effect of $r$ is also evident in this figure.

\subsection{Rosenbrock function}
The Rosenbrock function is a non-convex function defined by
$$f(x_1,x_2)=(1-x_1)^2 + 100(x_2-x_1^2)^2,$$
with a global minimum of $0$, occurring at $\mathbf{x}^\star=(1,1)$. This function is often used to assess the performance of mathematical optimization algorithms. In this experiment, we look at the convergence of gradient descent to this minimum for RN and $\SR_{p,r}$. The update rule of gradient descent is
$$
\mathbf{x}_{k+1} = \mathbf{x}_k - t_k \nabla f(\mathbf{x}_k),
$$
where $\mathbf{x}_k$ is the current point, $t_k$ is the learning rate, and $\nabla f(\mathbf{x}_k)$ is the gradient of the function at $\mathbf{x}_k$.
\textcolor{blue}{
We assume gradient computation using \texttt{binary64} with RN (gradients are then downcast to \texttt{binary16}) and apply $\SR_{p,r}$ only during the parameter update process.
}

\begin{figure}[t]
  \begin{center}
  \footnotesize
    \begin{tikzpicture}[trim axis group left,
        trim axis group right]
      \begin{groupplot}[
        group style={
          group size=2 by 1,
        },
        ymode=log,
        ymin = 1e-3,
        ymax = 1,
        width=2.5in,
        grid=major,
        every axis plot/.append style={mark repeat=5, line width=\mylinewidth}
        ]

        \nextgroupplot[
        ylabel={$|f(\mathbf{x}_k)-f(\mathbf{x}^\star)|$},
        title={$\mathbf{x}_0 = (0, 0)$},
        xlabel = {$k$},
        cycle list name = list_fig3
        ]

        \addplot table [x=n, y=min_binary-64, col sep=comma] {min_binary-64.csv};

        \addplot table [x=n, y=min_binary-61, col sep=comma] {rn16_min.csv};

        \addplot table [x=n, y=sr_min-3, col sep=comma] {sr_min-3.csv};

        \addplot table [x=n, y=sr_min-6, col sep=comma] {sr_min-6.csv};

        \addplot table [x=n, y=sr_min-7, col sep=comma] {sr_min-7.csv};

        \addplot table [x=n, y=sr_min-8, col sep=comma] {sr_min-8.csv};

        \addplot table [x=n, y=sr_min-10, col sep=comma] {sr_min-10.csv};

        \nextgroupplot[
        title={$\mathbf{x}_0 = (0.5, 0.5)$},
        xlabel = {$k$},
        cycle list name = list_fig3
        ]

        \addplot table [x=n, y=min_binary-64, col sep=comma] {min-5_binary-64.csv};LIMITED

        \addplot table [x=n, y=min_binary-61, col sep=comma] {rn16_min-5.csv};

        \addplot table [x=n, y=sr_min-3, col sep=comma] {sr-5_min-3.csv};

        \addplot table [x=n, y=sr_min-6, col sep=comma] {sr-5_min-6.csv};

        \addplot table [x=n, y=sr_min-7, col sep=comma] {sr-5_min-7.csv};

        \addplot table [x=n, y=sr_min-8, col sep=comma] {sr-5_min-8.csv};

        \addplot table [x=n, y=sr_min-10, col sep=comma] {sr-5_min-10.csv};

      \end{groupplot}
    \end{tikzpicture}

    \begin{tikzpicture}[trim axis left, trim axis right]
      \begin{axis}[
        title = {},
        legend columns=4,
        scale only axis,
        width=1mm,
        hide axis,
        /tikz/every even column/.append style={column sep=0.6cm},
        legend style={at={(0,0)},anchor=center,draw=none,
          legend cell align={left},cells={line width=\mylinewidth}},
        legend image post style={sharp plot},
        legend cell align={left},
        cycle list name = list_fig3
        ]
        \addplot (0,0);
        \addplot (0,0);
        \addplot (0,0);
        \addplot (0,0);
        \addplot (0,0);
        \addplot (0,0);
        \addplot (0,0);
        \legend{\texttt{binary64} RN, \texttt{binary16} RN, $\SR_{11,3}$, $\SR_{11,6}$, $\SR_{11,7}$, $\SR_{11,8}$, $\SR_{11,10}$};
      \end{axis}
    \end{tikzpicture}
  \end{center}
  \caption{Convergence profiles for 5,000 iterations of gradient descent on the Rosenbrock function. The parameter updates are performed using \texttt{binary64} arithmetic with RN and \texttt{binary16} arithmetic with RN and $\SR_{11,r}$. The starting value for the iteration is $\mathbf{x}_0 = (0, 0)$ for the profiles on the left and $\mathbf{x}_0 = (0.5, 0.5)$ for those on the right. For both experiments, we average each $\SR_{11,r}$ error over $500$ different runs, and the learning rate is $t_k=0.001$.}
  \label{fig:rosenbrock_errors}
\end{figure}
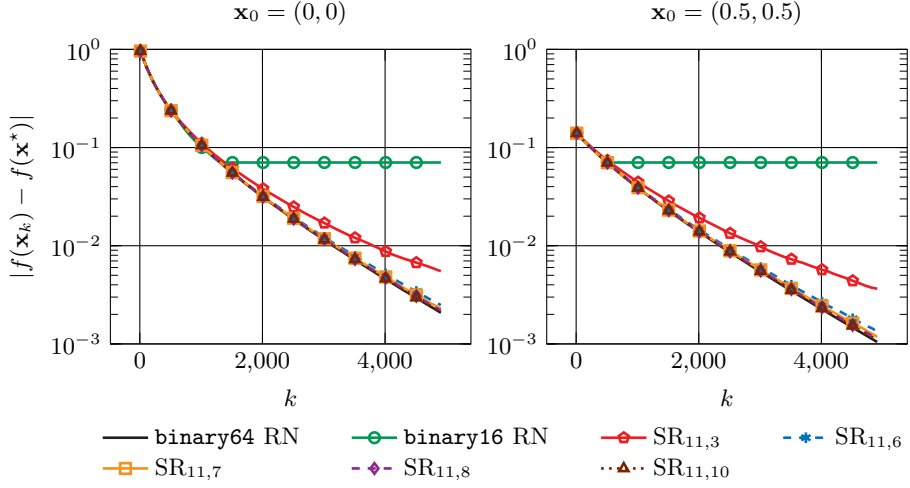

Figure~\ref{fig:rosenbrock_errors} illustrates the effect of $r$ on the convergence of the first 5,000 iterations of gradient descent for two different starting values. In the left panel, the starting value for the iteration is $\mathbf{x}_0 = (0, 0)$, for which $f(\mathbf{x}_0) = 1$ and $\nabla f(\mathbf{x}_0) = (-2, 0)$. This choice of $\mathbf{x}_0$ is close to the narrow, flat valley where the minimum lies, but the point is not too close to the minimum itself, which guarantees the need for a large number of iterations. In \texttt{binary16} arithmetic, RN and $\SR_{11,3}$ are both significantly affected by rounding errors, with RN stagnating rather early on. For larger values of $r$, the SR results are much better, matching those of \texttt{binary64} arithmetic with RN.

In the right panel, the starting value for the iteration is $\mathbf{x}_0 = (0.5,0.5)$, at which $f(\mathbf{x}_0)= 25.25$ and $\nabla f(\mathbf{x}_0)=(-50.5,50)$. This choice of which $\mathbf{x}_0$ lies in a region with higher curvature but closer to the minimum. While all implementations we consider do make progress initially, as soon as the iterates get close to $(1,1)$, \texttt{binary16} RN stagnates, whereas the limited-precision SR alternatives are not affected and continue to progress. As in the previous case, larger values of $r$ recover the baseline \texttt{binary64} convergence profile.

In both cases, as soon as $r$ is close to $\lceil \log_2(5{,}000) / 2\rceil=7$ the accuracy improvement of $\SR_{11,r}$ starts to plateau: the difference between $r=6$ and $r=7$ is somewhat visible, whereas $r=7$ and $r=8$ yield almost identical curves. This indicates that increasing $r$ above 7 leads to diminishing returns.

\subsection{Parameter update in deep neural network training} SR has found two main uses in modern deep learning scenarios: as a quantization procedure for network parameters and signals, in particular gradient signals~\cite{cbhy21,gagn15,wlcs18}, and to avoid stagnation during parameter updates~\cite{wcbc18,zzad20} when low-precision formats (16 and 8 bits) are used to compute and store them. The smaller memory footprint associated with lower precision formats is one of the key factors that has enabled the current boom of large language models.

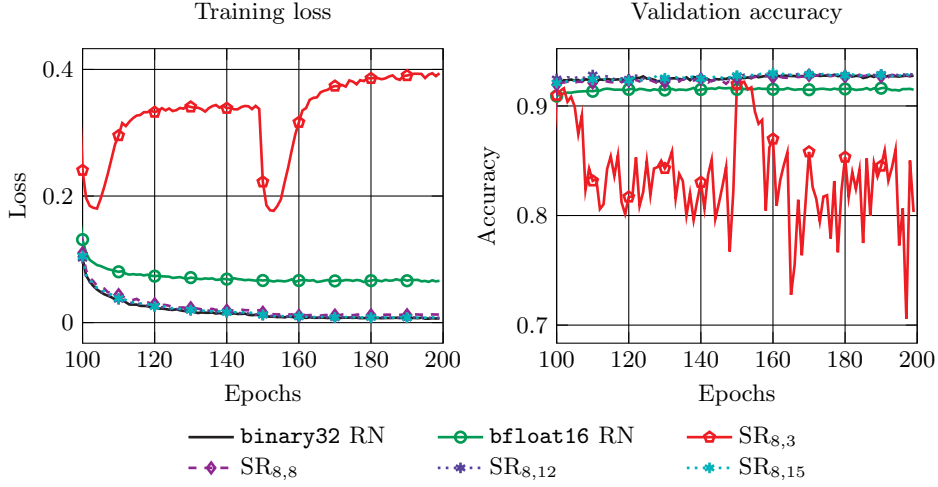
\begin{figure}[t]
  \begin{center}
  \footnotesize
    \begin{tikzpicture}[trim axis group left,
    trim axis group right]
      \begin{groupplot}[
        group style={
          group size=2 by 1,
          horizontal sep=1.5cm,
        },
        width=2.5in,
        grid=major,
        every axis plot/.append style={mark repeat=10, line width=\mylinewidth},
        ]

        \nextgroupplot[
        ylabel={Loss},
        title={Training loss},
        xlabel = {Epochs},
        xmin = 100,
        xmax = 200,
        ylabel near ticks,
        cycle list name = list_fig4
        ]

        \addplot table [x=Step, y=resnet32_binary32 - train_loss, col sep=comma] {resnet32_binary32_train_loss.csv};

        \addplot table [x=Step, y=resnet32_RN - train_loss, col sep=comma] {resnet32_RN_train_loss.csv};

        \addplot table [x=Step, y=resnet32_SR_3 - train_loss, col sep=comma] {resnet32_SR3_train_loss.csv};

        \addplot table [x=Step, y=resnet32_SR_8 - train_loss, col sep=comma] {resnet32_SR8_train_loss.csv};

        \addplot table [x=Step, y=resnet32_SR_12 - train_loss, col sep=comma] {resnet32_SR12_train_loss.csv};

        \addplot table [x=Step, y=resnet32_SR_15 - train_loss, col sep=comma] {resnet32_SR15_train_loss.csv};

        \nextgroupplot[
        title={Validation accuracy},
        xlabel = {Epochs},
        ylabel = {Accuracy},
        xmin = 100,
        xmax = 200,
        ylabel near ticks,
        cycle list name = list_fig4
        ]

        \addplot table [x=Step, y=resnet32_binary32 - test_acc, col sep=comma] {resnet32_binary32_test_acc.csv};

        \addplot table [x=Step, y=resnet32_RN - test_acc, col sep=comma] {resnet32_RN_test_acc.csv};

        \addplot table [x=Step, y=resnet32_SR_3 - test_acc, col sep=comma] {resnet32_SR3_test_acc.csv};

        \addplot table [x=Step, y=resnet32_SR_8 - test_acc, col sep=comma] {resnet32_SR8_test_acc.csv};

        \addplot table [x=Step, y=resnet32_SR_12 - test_acc, col sep=comma] {resnet32_SR12_test_acc.csv};

        \addplot table [x=Step, y=resnet32_SR_15 - test_acc, col sep=comma] {resnet32_SR15_test_acc.csv};

      \end{groupplot}
    \end{tikzpicture}

    \begin{tikzpicture}
      \begin{axis}[
        title = {},
        legend columns=3,
        scale only axis,
        width=1mm,
        hide axis,
        /tikz/every even column/.append style={column sep=0.6cm},
        legend style={at={(0,0)},anchor=center,draw=none,
          legend cell align={left},cells={line width=\mylinewidth}},
        legend image post style={sharp plot},
        legend cell align={left},
        cycle list name = list_fig4
        ]
        \addplot (0,0);
        \addplot (0,0);
        \addplot (0,0);
        \addplot (0,0);
        \addplot (0,0);
        \addplot (0,0);
        \legend{\texttt{binary32} RN, \texttt{bfloat16} RN, $\SR_{8,3}$, $\SR_{8,8}$, $\SR_{8,12}$, $\SR_{8,15}$};
      \end{axis}
    \end{tikzpicture}
  \end{center}
    \caption{Evolution of training loss (left) and validation accuracy (right) for a ResNet32 image classification network on the CIFAR-10 dataset using various parameter and optimizer update configurations. In the baseline configuration,\texttt{binary32} arithmetic with RN is used for compute and the same format is used for storage. For the low-precision configurations, parameters are stored and updated using \texttt{bfloat16} arithmetic with either RN or $\SR_{p,r}$.}
    \label{fig:resnet32_comp}
\end{figure}

Here, we focus on the second case and consider parameter updates during the training of a ResNet32 image classification model~\cite{hzrs16} on the CIFAR-10~\cite{kriz09} dataset. The training hyperparameters are identical to those used in~\cite[sect.~4.2]{hzrs16}: minibatch gradient descent with 128 batch size, momentum set to $0.9$, weight decay to $10^{-4}$, $0.1$ learning rate that gets divided by $10$ after 32,000 and 48,000 iterations, and $n=64{,}000$ iterations ($200$ epochs) of training.
We use \texttt{bfloat16} arithmetic, for which $p = 8$, and the update rule for the network parameters $\mathbf{x}$ in this configuration is
\begin{eqnarray*}
    \mathbf{v}_{k+1} &=& \circ(\mu \mathbf{v}_k + \mathbf{g}_k), \\
    \mathbf{x}_{k+1} &=& \circ(\mathbf{x}_k - t_k\mathbf{v}_{k+1}),
\end{eqnarray*}
where $\mu$ is the momentum term, $\mathbf{v}_k$ is the velocity vector, $\mathbf{g}_k$ is the gradient of the loss function with respect to $\mathbf{x}_k$, and $\circ$ is the rounding operator, which can round to \texttt{bfloat16} using either RN or $\SR_{8,r}$, depending on the configuration.

The results of this experiment are reported in~\Cref{fig:resnet32_comp}. For \texttt{bfloat16}, the training loss is degraded when using RN or $\SR_{8,r}$ with very small $r$, but larger values of $r$ match the baseline \texttt{binary32} results on both the training and the validation datasets: the validation accuracy of \texttt{binary32} is $92.85\%$ but goes down to $91.68\%$ for \texttt{bfloat16} with RN. For \texttt{bfloat16} with $\SR_{p,r}$, the choice $r=3$ leads to an unstable validation accuracy that hovers around $83\%$, but using $\SR_{8,r}$ with $r\geq\lceil \log_2(64{,}000) / 2\rceil=8$ recovers baseline accuracy, with $92.86\%, 92.93\%$, and $92.99\%$ for $r=8$, $r=12$, and $r=15$, respectively. Diminishing improvements are again visible as $r$ is increased beyond this threshold.


\section{Conclusions}
\label{sec:conclusions}

SR has garnered significant attention in various domains~\cite{cfhm22}, as it can deliver improved accuracy compared with the deterministic rounding modes in the IEEE 754 standard~\cite{ieee19}. We have investigated the number of random bits required to implement SR effectively. We introduced a new rounding mode, limited-precision SR, denoted by $\SR_{p,r}$, and we showed that the bias that this rounding mode introduces depends on the value of $r$ used. In \Cref{lem:main-result}, we presented a model that facilitates the theoretical analysis of algorithms using $\SR_{p,r}$. With this lemma, we can analyze all algorithms that were previously studied for classical SR. Applying this lemma to recursive summation and inner product computation, we derived probabilistic error bounds \textcolor{blue}{proportional to $\sqrt{n}u_p + nu_{p+r}$.}

Our findings suggest that $\SR_{p,r}$ becomes less biased as the number of random bits $r$ increases, and that for a large enough value it converges to the theoretical properties of $\SR_p$. On the other hand, using a large value of $r$ may not always be practical or necessary, as the computational overhead associated with generating and processing a large number of random bits can be significant. Therefore, it is crucial to strike a balance between the desired accuracy and the computational resources available when implementing SR in applications. To this end, our bounds suggest that choosing a value of $r$ that is close to $\lceil (\log_2n) / 2\rceil$ offers the best compromise in practice \textcolor{blue}{for summation and dot product operations}. This has been verified through several numerical experiments: recursive summation, gradient descent on the Rosenbrock function, and parameter updates in the training of deep learning models.



\ifx\ispreprint\undefined
\bibliographystyle{siamplain}
\else
\bibliographystyle{plain-doi}
\fi
\bibliography{references}
\end{document}